\documentclass[11pt, a4paper, oneside]{amsart}

\usepackage[english]{babel}
\usepackage{amsmath, amsthm, amsfonts, mathrsfs, amssymb}
\usepackage{mathtools}
\mathtoolsset{centercolon}
\usepackage{mathrsfs}
\usepackage{amscd}
\usepackage{bm}

\usepackage{tikz-cd}

\usepackage[shortlabels]{enumitem}
\usepackage[colorlinks, citecolor = blue, urlcolor={red}]{hyperref}
\usepackage[colorinlistoftodos,prependcaption,textsize=small]{todonotes}

\usepackage{faktor}

\allowdisplaybreaks

\setlist[itemize]{leftmargin=25pt}
\setlist[enumerate]{leftmargin=25pt}

\usepackage{color}

\newtheorem{theorem}{Theorem}
\newtheorem{corollary}[theorem]{Corollary}
\newtheorem{lemma}[theorem]{Lemma}
\newtheorem{proposition}[theorem]{Proposition}

\theoremstyle{remark}
\newtheorem{remark}[theorem]{Remark}

\theoremstyle{definition}
\newtheorem{definition}[theorem]{Definition}
\newtheorem{example}[theorem]{Example}

\numberwithin{theorem}{section}
\numberwithin{equation}{section}

\def\N{{\mathbb N}}

\def\R{{\mathbb R}}

\def\F{{\mathbb F}}

\DeclareMathAlphabet{\mathpzc}{OT1}{pzc}{m}{it}

\newcommand{\one}{{{\bf 1}}}
\newcommand{\hra}{\hookrightarrow}

\newcommand{\wh}{\widehat}
\newcommand{\supp}{\text{\rm supp\,}}

\newcommand{\norm}[1]{\nrm{#1}}

\renewcommand{\d}{\mathrm{d}}

\renewcommand{\tilde}{\widetilde}

\newcommand{\Sch}{\mathscr{S}}

\newcommand{\Schw}{{\mathscr S}}

\newcommand{\dom}{\operatorname{D}}
\newcommand{\loc}{\mathrm{loc}}

\newcommand{\p}{\partial} 
\newcommand{\graph}{\operatorname{graph}} 
\newcommand{\sign}{\operatorname{sign}} 

\DeclareFontFamily{U}{mathx}{\hyphenchar\font45}
\DeclareFontShape{U}{mathx}{m}{n}{<5> <6> <7> <8> <9> <10> <10.95> <12> <14.4> <17.28> <20.74> <24.88> mathx10}{}
\DeclareSymbolFont{mathx}{U}{mathx}{m}{n}
\DeclareFontSubstitution{U}{mathx}{m}{n}
\DeclareMathAccent{\widecheck}{0}{mathx}{"71}

\newcommand{\ms}{\mathscr}

\newcommand{\mrm}{\mathrm}

\DeclarePairedDelimiter{\ip}\langle\rangle
\DeclarePairedDelimiter{\nrm}\lVert\rVert

\newcommand{\ipb}[1]{\bigl\langle#1\bigr\rangle}

\title[]{Cauchy problem for singular-degenerate porous medium type equations: well-posedness and Sobolev regularity}

\author{Nick Lindemulder}
\address[N. Lindemulder]{
Radboud Universiteit\\
IMAPP - Mathematics\\
P.O. Box 9010 \\
6500 GL Nijmegen\\
The Netherlands}
\email{nick.lindemulder@gmail.com}

\author{Stefanie Sonner}
\address[S. Sonner]{
Radboud Universiteit\\
IMAPP - Mathematics\\
P.O. Box 9010 \\
6500 GL Nijmegen\\
The Netherlands}
\email{stefanie.sonner@ru.nl}



\begin{document}

\maketitle

\begin{abstract}
Motivated by models for biofilm growth, we consider Cauchy problems for quasilinear reaction diffusion equations where the diffusion coefficient has a porous medium type degeneracy as well as a singularity. We prove results on the well-posedness and Sobolev regularity of solutions. The proofs are based on m-accretive operator theory, kinetic formulations and Fourier analytic techniques.     
\end{abstract}

\section{Introduction and main results}

We prove the well-posedness and investigate Sobolev regularity for the following Cauchy problem for singular-degenerate porous medium type equations,
\begin{align}\label{eq:cauchy}
    \partial_tu&=\Delta \phi(u)+f(u)&&\text{in}\ (0,T)\times\R^d,\\
    u|_{t=0}&=u_0&&\text{on}\ \R^d,   \nonumber 
\end{align}
where the solution $u$ takes values in $[0,1)$ and $\phi:[0,1)\to\R$ is a strictly increasing function with $\phi(0)=0$   that has a degeneracy $\phi'(0)=0$ and singularity 
$\phi(1)=\infty$.
Moreover, the function $f$ is Lipschitz continuous, $f(0)=0$ and the initial data satisfies $u_0\in L^1(\R^d;[0,1))$.

The problem is motivated by models for biofilm growth where a degenerate equation of the form \eqref{eq:cauchy} is coupled to an additional semilinear PDE or ODE for the nutrient concentration  \cite{EbJaDuWo,EbPaLo}. We refer to these two different settings as PDE-PDE or ODE-PDE model, respectively, below. 
In such models, the solution $u$ describes the time evolution of the biomass fraction, i.e. the normalized biomass density,
and the actual biofilm is represented by the region $\{x\in\R^d: u(t,x)>0\}.$
The biomass diffusion coefficient is of the form 
\begin{align}\label{eq:bfdiff}
\phi'(u)=\frac{u^b}{(1-u)^a}, \qquad a\geq 1, b>0,
\end{align}
and the function $f$ models biomass production.
The degeneracy $u^b$ in \eqref{eq:bfdiff} is well-known from the porous medium equation. It leads to the formation of free boundaries and enforces a finite speed of propagation, i.e. solutions emanating from initial data with compact support remain compactly supported for all times. The additional singularity $(1-u)^{-a}$ in \eqref{eq:bfdiff} ensures that solutions remain bounded by 1 despite the growth term $f$ in \eqref{eq:cauchy}. 

The well-posedness for the PDE-PDE biofilm growth model \cite{EbPaLo} in bounded domains with Dirichlet boundary conditions was established in \cite{EZE09}. The solution theory was extended in \cite{HMS22,MS23} for more general PDE-PDE and ODE-PDE systems involving a degenerate equation of the form \eqref{eq:cauchy} and allowing for mixed Dirichlet-Neumann boundary conditions which are relevant in applications. The local H\"older regularity of solutions for such PDE-PDE systems was shown in \cite{HM22a}. For coupled ODE-PDE systems the existence of traveling wave solutions in one spatial dimension was proven in  \cite{MHSED23}. 
The Cauchy problem for PDE-PDE systems with the specific diffusion coefficient $\phi'(u)=\frac{u^b}{(1-u)^a}$ was considered in \cite{HM22} where the existence and uniqueness of weak energy solutions was shown  by approximating the equation in $\R^d$ by problems on bounded domains. 

The aim of our paper is to develop a solution theory for the more general Cauchy problem \eqref{eq:cauchy} based on mild solutions and the theory of $m$-accretive operators, to show well-posedness  and to derive space-time regularity results in the scale of Sobolev spaces. For the latter we use kinetic formulations and velocity averaging inspired by the approach applied to the porous medium equation in \cite{GST20,TT07}.  While in models for biofilm growth solutions take non-negative values we formulate our results in a broader framework and allow for sign changing solutions. In the sequel, we assume that the solution $u$ takes values in $(-1,1)$, that $\phi:(-1,1)\to\R$ is strictly increasing and satisfies $\phi(0)=\phi'(0)=0$, $\phi(-1)=-\infty$ and $\phi(1)=\infty.$ The first main result addresses the well-posedness for problem \eqref{eq:cauchy}. In \cite{GST20} and related earlier works on the Sobolev regularity of solutions \cite{Ge21,CP03,TT07} the concept of entropy solutions was used while we base our analysis on the notion of  mild solutions and the theory of $m$-accretive operators. To avoid technical assumptions the second statement of Theorem \ref{thm:wellposedness;main1}
is formulated for the specific biofilm diffusion coefficient \eqref{eq:bfdiff} but we prove it in greater generality in Section \ref{sec:well} and Section \ref{sec:comp}.

\begin{theorem}\label{thm:wellposedness;main1}
Let $\phi \colon (-1,1) \to \R$ be a maximal monotone function with $\phi(0)=0$ and let $f\colon[-1,1] \to \R$ be a Lipschitz function with $f(0)=0$. Then for every $u_0\in L^1(\R^d;[-1,1])$ there exists a unique mild solution $u \in C([0,T];L^1(\R^d;[-1,1]))$ of the initial value problem
\begin{align}\label{eq:GPME}
\begin{split}
    \p_t u &= \Delta \phi(u)+f(u) \qquad\text{in}\:\:(0,T) \times \R^d, \\
u(0)&=u_{0} \qquad \text{on}\:\:\R^d.    
\end{split}
\end{align}
If $d\geq3$ and $\phi'(u)=\frac{|u|^b}{(1-|u|)^a},  a\geq 1, b>0,$ then $u$ is a distributional solution of \eqref{eq:GPME} and 
$u(t,x)\in(-1,1)$ for all $t\in(0,T)$ and almost every $x\in\R^d$.
Moreover, if $u_0\in[0,1)$ then $u(t,x)\in[0,1)$ for all $t\in(0,T)$ and almost every $x\in\R^d$.
\end{theorem}

Our second main result provides space time regularity in the scale of Sobolev spaces. 
In addition to the hypotheses of Theorem \ref{thm:wellposedness;main1} we need to assume that the degeneracy of $\phi$ in $0$ is of porous medium type. Again, to avoid technicalities we formulate Theorem \ref{thm:regularity} for the specific case \eqref{eq:bfdiff} but prove the result under more general assumptions in Section \ref{sec:reg}.

\begin{theorem}\label{thm:regularity}
    Let the hypotheses of Theorem \ref{thm:wellposedness;main1} hold, let $d\geq3$ and  
     $\phi'(u)=\frac{|u|^b}{(1-|u|)^a},  a\geq 1, b\geq 1$.
Let  $u \in C(0,T];L^1(\R^d;[-1,1]))$ be the unique mild solution of \eqref{eq:GPME}.    For $p \in [2,b+1]$ we define
    $$
    \kappa_t := \frac{b+1-p}{pb},\quad \kappa_x := \frac{2(p-1)}{pb}.
    $$
    Then, for all $\sigma_t \in [0,\kappa_t) \cup \{0\}$ and $\sigma_x \in [0,\kappa_x)$, we have 
    $$
    u \in W^{\sigma_t,p}(0,T;W^{\sigma_x,p}(\R^d)).
    $$
\end{theorem}

To the authors' knowledge, earlier well-posedness and regularity results for Cauchy problems for  generalized porous medium equations do not cover Problem \eqref{eq:GPME} where $\phi'$ has an additional singularity and the equation includes a growth term $f$. 
While Theorem \ref{thm:wellposedness;main1} is based on mild solutions, the proof of Theorem \ref{thm:regularity} requires kinetic formulations in order to employ velocity averaging techniques as done for porous medium type equations in \cite{GST20,Ge21,GS23}.
We refer to \cite[Chapter 10]{Va07} for the history and an overview of the literature on the semigroup approach and mild solutions to generalized porous medium equations.  
A detailed overview of Sobolev regularity results for nonlinear evolution equations based on velocity-averaging techniques
 is given in \cite{Ge21}.
Regularity results in the scale of Sobolev spaces as stated in Theorem \ref{thm:regularity} 
were shown for non-local porous medium type operators in \cite{GS23} 
which in turn are generalizations of earlier works on the porous medium equation \cite{GST20,Ge21}.

The outline of the paper is as follows. In Section \ref{sec:pre} we introduce notation and function spaces. We also recall existence results for classical solutions of non-degenerate quasilinear problems as well as facts from the theory of $m$-accretive operators. In Section \ref{sec:well} we prove existence and uniqueness for problem \eqref{eq:GPME} and show that solutions take value in the open interval $(-1,1)$. In Section \ref{sec:approx} we establish stability results and show that the concepts of mild and classical solutions coincide for smooth non-degenerate approximations. 
We also use the approximations to prove comparison results and show that \eqref{eq:GPME} preserves the non-negativity of solutions. 
Finally, Section \ref{sec:reg} is devoted to kinetic formulations and the proof of the Sobolev regularity of solutions.

\subsection*{Acknowledgments }
The authors would like thank Jonas Sauer for his visit to Radboud University Nijmegen, the insightful discussions and for  kindly giving them a preliminary unpublished draft of \cite{GS23}. 
They also thank the Nederlandse Organisatie voor  Wetenschappelijk
Onderzoek (NWO) for their support through the Grant OCENW.KLEIN.358.

\section{Preliminaries}\label{sec:pre}

In this section we introduce notation and relevant function spaces. Moreover, we summarize results from the theory of uniformly parabolic quasilinear problems and $m$-accretive operators that we will need in the sequel.

\subsection{Notation and  function spaces}\label{subsec:function_spaces}

We write $f\lesssim g$ if there is a constant $c>0$ such that $f\leq c g$  and 
$f\lesssim_\varepsilon g$ if there is a  constant $c_\varepsilon\geq 0$ depending on the parameter $\varepsilon$ such that $f\leq c_\varepsilon g$. Moreover, we use the notation 
$f\eqsim g$ if $f\lesssim g$ and $g\lesssim f$.
For sets $I,J\subset\R$ we write $J\Subset I$ if $J$ is compactly contained in $I$. We introduce the functions 
$$
\sign_0(x)=\begin{cases}
    -1&x<0,\\
    0&x=0,\\
    1&x>0,
\end{cases}\qquad 
\sign_0^+(x)=\begin{cases}
    0&x\leq 0,\\
    1&x>0,
\end{cases}\qquad x\in\R.
$$
We write $\N_0 =\{0\}\cup\N$, where $\N=\{1,2,3,\ldots\}$.

Let $V \subset \R^d$ have dense interior (e.g.\ $V=J$ an interval in $\R$ or $V\subset\R^d$ open) and let $X$ be a Banach space. 
For $k\in\N_0\cup\{\infty\}$ we denote by $C^k(V;X)$ the space of all $k$-times continuously differentiable functions $V \to X$.
We denote by $BC^k(V;X)$ the subspace of all bounded functions with bounded derivatives up to order $k$.
For $k \in \N_0$ this is a Banach space with the norm $\|u\|_{BC^k(V;X)}=\sup_{|\alpha| \leq k}\sup_{x\in V}\|\p^\alpha u(x)\|_X$.
Similarly, we use the notation $BC^\alpha(V;X)$, $\alpha\in(0,\infty)\setminus \N$, for spaces of H\"older continuous functions.
If $k=0$, we write $C(V;X)$ and $BC(V;X)$.
If $X=\R$ we write $C^k(V)$, $BC^k(V)$ and $BC^\alpha(V)$.
Furthermore, we denote by $C_c(V)^+$ the continuous real-valued functions on $V$ with compact support that take non-negative values.
Finally, for an interval $I\subset\R$ we use the notation 
$$BC^{\alpha,\beta}(I\times\R^d)=
BC^{\alpha}(I;BC(\R^d))\cap
BC(I;BC^\beta(\R^d)),\quad \alpha,\beta\in[0,\infty).
$$
Let $\Omega \subset \R^d$ be open. We denote by $\ms{D}(\Omega)$ the space $C^\infty_c(\Omega)$ equipped with its standard inductive limit topology and by $\ms{D}'(\Omega;X) = \ms{L}(\ms{D}(\Omega);X)$ the space of $X$-valued distributions on $\Omega$. 
We denote by $\ms{S}(\R^d)$ the space of Schwartz functions with its standard Frechet topology and by $\ms{S}'(\R^d;X) = \ms{L}(\ms{S}(\R^d);X)$ the space of $X$-valued tempered distributions.

For $1\leq p\leq \infty$ we denote by $L^p(\Omega;X)$ the standard Bochner space, and by $W^{k,p}(\Omega;X),$ $k\in\N_0,$ the corresponding $X$-valued Sobolev space.



We denote the integrable functions taking values in $A\subset\R$ by $L^1(\Omega;A)=\{f \in L^1(\Omega): f \in A \:\text{a.e.}\}$.
Moreover, let $1\leq p< \infty$  and $s \in (0,\infty) \setminus \N$, say $s=k+\sigma$ with $k \in \N_0$ and $\sigma \in (0,1)$. 
We define the \emph{Sobolev-Slobodetskii space} $W^{s,p}(\Omega;X)$ as the space of all $f \in W^{k,p}(\Omega;X)$ for which
\begin{align}\label{eq:SobSlob_norm}
[f]_{\dot W^{s,p}(\Omega;X)} :=  \sup_{|\alpha| = k} \left(\int_{\Omega \times \Omega}\frac{\nrm{\p^\alpha f(x) - \p^\alpha f(y) }_{X}^p}{|x-y|^{\sigma p + d}} \d x \d y\right)^{1/p} < \infty.
\end{align}
Equipped with the norm
$$
\nrm{f}_{W^{s,p}(\Omega;X)} := \nrm{f}_{W^{k,p}(\Omega)} + [f]_{\dot W^{s,p}(\Omega;X)},
$$
$W^{s,p}(\Omega;X)$ becomes a Banach space.

There are several ways to define  \emph{homogeneous Sobolev spaces}. We only consider the spaces $\dot H^{1}(\R^d)$ and $\dot H^{-1}(\R^d)$, and 
only use them in the case $d \geq 3$ for which Sobolev embeddings exist.
To simplify the presentation, we actually base our definition on these embeddings.
Let $d \geq 3$ and 
$p \in (1,2)$ be such that $\frac{1}{p}=\frac{1}{2}+\frac{1}{d}$, or equivalently $\frac{1}{p'}=\frac{1}{2}-\frac{1}{d}$, that is, $p'$ is the Sobolev conjugate of $2$.
Defining
$$
\dot{H}^{1}(\R^d) := \{ u \in L^{p'}(\R^d) : \nabla u \in L^2(\R^d) \},
$$
we have
\begin{equation}\label{eq:Sob_embd_H^-1;dual}
\dot{H}^{1}(\R^d) \hra L^{p'}(\R^d),   
\end{equation}
as a consequence of the Hardy-Littlewood-Sobolev theorem on fractional integration (see e.g.\ \cite[Theorem~6.1.3]{Gr09}).
As this embedding is dense, defining $\dot{H}^{-1}(\R^d) := \dot{H}^{1}(\R^d)^{*}$ by duality we obtain
\begin{equation}\label{eq:Sob_embd_H^-1}
L^p(\R^d) \hookrightarrow  \dot{H}^{-1}(\R^d).
\end{equation}
In this way, $\dot H^{1}(\R^d)$ and $\dot H^{-1}(\R^d)$ are Hilbert spaces.

Alternatively to  defining Sobolev-Slobodetskii spaces via the difference norm description \eqref{eq:SobSlob_norm}, there is a Fourier analytic description which gives rise to the more general scale of Besov spaces.
In order to define Besov spaces,
we introduce the following notation. 
We denote by $\Phi(\R^d)$ the set of all sequences $(\varphi_k)_{k\in \N_0} \subset \ms{S}(\R^d)$ such that for $k\geq 2$,
\begin{align*}
\wh{\varphi}_0 = \wh{\varphi}, \qquad \wh{\varphi}_1(\xi) = \wh{\varphi}(\xi/2) - \wh{\varphi}(\xi), \qquad \wh{\varphi}_k(\xi) = \wh{\varphi}_1(2^{-k+1} \xi), \qquad \xi\in \R^d,
\end{align*}
where the Fourier transform $\wh{\varphi}$ of the generating function $\varphi\in \Sch(\R^d)$ satisfies
\begin{equation*}
 0\leq \wh{\varphi}(\xi)\leq 1, \quad  \xi\in \R^d, \qquad  \wh{\varphi}(\xi) = 1 \ \text{ if } \ |\xi|\leq 1, \qquad  \wh{\varphi}(\xi)=0 \ \text{ if } \ |\xi|\geq \frac32.
\end{equation*}

Let $p \in [1,\infty)$, $q\in [1,\infty]$ and $s \in \R$. 
We define the \emph{Besov space} $B_{p,q}^s (\R^d;X)$ as the space of all $f\in {\mathscr S}'(\R^d;X)$ for which
$$
\|f\|_{B_{p,q}^s (\R^d;X)} := \Big\| \big( 2^{sk}\varphi_k * f\big)_{k \in \N_0} \Big\|_{\ell^q(L^p(\R^d;X))} < \infty,
$$
where $(\varphi_k)_{k\in \N_0} \in \Phi(\R^d)$ is fixed.
With this norm $B_{p,q}^s (\R^d;X)$ is a Banach space. The definition is independent of the sequence $(\varphi_k)_{k\in \N_0} \in \Phi(\R^d)$ in the sense that a different choice of $(\varphi_k)_{k\in \N_0}$ leads to an equivalent norm. 
We remark that for $p \in (1,\infty)$ and $s \in (0,\infty) \setminus \N$, we have
\[
W^{s,p}(\R^d;X) = B_{p,p}^s (\R^d;X).
\]

For a measurable space $(S,\ms{A})$  we denote by $\ms{M}(S;X)$ the space of all $X$-valued measures on $(S,\ms{A})$ of bounded variation. 
We refer the reader to \cite[Section~1.3b]{HNVW16} for a brief introduction to vector measures.
If $X=\R$, we write $\ms{M}(S)$.
Let now $(S,\ms{A},\mu)$ be a measure space. 
Following \cite[Section~2.4]{Pi16}, we denote by $\Lambda^\infty(S;X) = \Lambda^\infty((S,\ms{A},\mu);X)$ the space of all $X$-valued measures $F:\ms{A} \to X$ that are absolutely continuous with respect to $\mu$ and for which the Radon-Nikod\'ym derivative $\frac{\d\nrm{F}}{\d\mu}$ belongs to $L^\infty(S)$.  
Equipped with the norm
$$
\nrm{F}_{\Lambda^\infty(S;X)} := \left\|\frac{\d\nrm{F}}{\d\mu}\right\|_{L^\infty(S)},
$$
$\Lambda^\infty(S;X)$ becomes a Banach space. 
Furthermore, we denote by $L^0(S)$ the space of equivalence classes of measurable functions on $(S,\ms{A},\mu)$.

\subsection{Classical solutions of  quasilinear parabolic problems}
We consider the following Cauchy problem
\begin{equation}\label{eq:GPME;psi_v}
\begin{cases}
\p_t v = \Delta \psi(v)+g &\text{in}\:\:(0,T) \times \R^d, \\
v(0)=v_{0} &\text{on}\:\:\R^d.    
\end{cases}
\end{equation}
Under suitable assumptions on $\psi$, $g$ and $v_0$
the existence and uniqueness of classical solutions follows from \cite{LSU68} 
(Theorem 8.1, Chapter V).  

\begin{proposition}\label{pro_LSU}
    Let $\psi\in C^3(\R)$ such that $\psi',\psi''\in L^\infty(\R)$ and $\psi'\geq c>0$ and let $\beta \in (0,1)$.
    Furthermore, we assume that $g\in BC^{\frac{\beta}{2}, \beta}([0,T]\times \R^d)$, $v_0\in BC^{2+\beta}(\R^d)$. 
    Then, there exists a unique classical solution $v$ of \eqref{eq:GPME;psi_v} and $v\in BC^{1+\frac{\beta}{2},2+\beta}([0,T]\times \R^d)$.
\end{proposition}

\begin{proof}
    Using the notation in \cite{LSU68}, 
    $$
    \mathcal{L}u=u_t-\sum_{i=1}^d\partial_{x_i}a_i(\cdot,u,\nabla u)+a(\cdot, u, \nabla u),
    $$ 
    we have $a_i(\cdot, u, \nabla u)=\psi'(u)\partial_{x_i}u$ and $a(\cdot, u, \nabla u)=-g(\cdot).$ Moreover, we observe that $a_{ij}(\cdot, u, \nabla u)=\psi''(u)\delta_{ij}$ which implies that 
    $$
    A(\cdot, u, p)=a(\cdot, u,p)-\sum_{i=1}^d\partial_u a_i(\cdot, u,p) p_i-\sum_{i=1}^d\partial_{x_i}a_i (\cdot, u,p)=-g(\cdot)-\psi''(u)|p|^2.
    $$
    We verify the assumptions of Theorem 8.1, Chapter 5, in \cite{LSU68}. Hypothesis (a) holds as the initial data $v_0\in  C^{2+\beta}(\R^d)$ is bounded. Hypothesis (b) holds since $a_{ii}(\cdot, u, \nabla u)=\psi'(u)>0$ and Young's inequality implies that 
    $$
    A(\cdot, u, 0)u=-g(\cdot)u\geq -\frac{1}{2}u^2-c_1,
    $$
    for some constant $c_1\geq 0$. To verify (c) we observe that the functions $a$ and $a_i$ are continuous and $a_i$ is continuously differentiable with respect to $x, u$ and $p$, and 
    $$
    0<c\leq \psi'(u)\leq c_2<\infty, \qquad c_2\in\R.
    $$
    Moreover, using Young's inequality and that $\psi',\psi''\in L^\infty (\R)$ and $g$ is bounded we observe that 
    \begin{align*}
        &\quad\sum_{i=1}^d\left(|a_i|+\left|\partial_u a_i\right| \right)(1+|p|)+\sum_{i,j=1}^d|\partial_{x_j}a_i|+|a|\\
        &= \sum_{i=1}^d \left(|\psi'(u) p_i|+\left|\psi''(u) p_i\right| \right)(1+|p|)+|g|
        \leq c_3(1+|p|)^2,
    \end{align*}
    for some $c_3>0$. The functions $a, a_i, \partial_{p_j} a_i$ and $\partial_u a_i$ are H\"older continuous with respect to $t,x,u$ and $p$ with exponents $\frac{\beta}{2},\beta,\beta$ and $\beta$. Moreover, all constants in the estimates are uniform. 
    Hence, Theorem 8.1, Chapter V in \cite{LSU68} implies that there exists a solution that is bounded and satisfies $u\in C^{1+\frac{\beta}{2},2+\beta}([0,T]\times R^d).$
    The uniqueness follows as well since the functions $a_{ij}$ and $A$ are differentiable with respect to $u$ and $p$, and for bounded values of $u$ and $p$, we have 
    \begin{align*}
        \max_{(x,t)\in\R^d\times[0,T], |(u,p)|\leq R}&\left\{\partial_u a_{ij}(t,x,u,p), 
        \partial_p a_{ij}(t,x,u,p),\partial_p A (t,x,u,p)\right\}\leq C_R,\\
        \min_{(t,x)\in\R^d\times[0,T], |(u,p)|\leq R}&\left\{\partial_u A (t,x,u,p)\right\}\geq -C_R,
    \end{align*}
    for some constant $C_R\geq 0$ only depending on $R>0$.   
\end{proof}

\subsection{Abstract nonlinear Cauchy problems governed by m-accretive operators}

Concerning \emph{maximal monotone} and \emph{$m$-accretive operators}
we follow the terminology from
\cite[Chapters~2 and 3]{Ba10} (also see \cite[Chapter~10]{Va07}), with the difference that we only consider single-valued operators. This is not necessary, but simplifies the presentation and is sufficient for our purposes.

\subsubsection{Maximal monotone functions}

We shortly introduce maximal monotone functions and graphs, for further details we refer to \cite{Ba10}. 

Let $I \subset \R$ be an open interval and let $\phi:I \to \R$ be a monotonically increasing function.
A function $\phi$ is   \emph{maximal monotone} if it satisfies:
\begin{itemize}
    \item if $\inf(I) > -\infty$, then $\inf_I \phi = -\infty$;
    \item if $\sup(I) < \infty$, then $\sup_I \phi = \infty$.
\end{itemize}

A \emph{monotone graph} in $\R$ is a set $\beta \subset \R \times \R$ such that, for all $(x_1,y_1), (x_2,y_2) \in \beta$, it holds true that $(x_1-x_2)(y_1-y_2) \geq 0$.
A monotone graph $\beta$ in $\R$ is called \emph{maximal monotone} if it is not properly contained in any other monotone graph in $\R$.

Let $I \subset \R$ be an open interval,  $\phi:I \to \R$ and consider the graph $\graph(\phi)=\{(r,\phi(r)):r \in I\}$. Then note that the following holds:
\begin{itemize}
    \item $\phi$ is monotonically increasing if and only if $\graph(\phi)$ is a monotone graph in $\R$.
    \item $\phi$ is a maximal monotone function if and only if $\graph(\phi)$ is a maximal monotone graph in $\R$.
\end{itemize}

\begin{example}\label{ex:max_mon_fc_biofilm}
Let $I=(-1,1)$, $a\geq 1, b>0$ and consider the biofilm diffusion coefficient \eqref{eq:bfdiff}, 
$$
D(z)=\frac{|z|^b}{(1-|z|)^a}, \qquad z \in I,
$$
and let
$$
\phi(\rho) = \int_{0}^{\rho}D(z)\d z, \qquad \rho \in I.
$$
Then $\phi:I \to \R$ is a maximal monotone function.
\end{example}
\begin{proof}
Since $D>0$  on $I \setminus \{0\}$, it follows that $\phi$ is strictly monotonically increasing. 
Furthermore, as $a \geq 1$ we have $-\lim_{\rho \to -1}\phi(\rho)=\lim_{\rho \to 1}\phi=\infty$.
Therefore, $\phi:I \to \R$ is a maximal monotone function.     
\end{proof}

\subsubsection{$m$-accretive operators}

The following theorem corresponds to \cite[Theorem~3.7]{Ba10}. It defines the operator  $A_\phi$ that will allow us to consider \eqref{eq:cauchy} as an abstract evolution equation and apply the theory of nonlinear semigroups.

\begin{proposition}\label{thm:m-accr_A_phi}
Let $I \subset \R$ be an open interval with $0 \in I$ and let $\phi:I \to \R$ be a maximal monotone function with $\phi(0)=0$. Then the operator $A_\phi$ on $L^1(\R^d)$ given by
\begin{align*}
    \dom(A_\phi) &= \{ y \in L^1(\R^d) : \phi(y) \in L^1_\loc(\R^d), \Delta\phi(y) \in L^1(\R^d) \},\\
    A_\phi y &= -\Delta\phi(y), \  y\in \dom(A_\phi),
\end{align*}
is m-accretive.
\end{proposition}

The closure of the domain of  $A_\phi$ has the following explicit characterization, which is taken from \cite[Proposition~6]{BC81}.

\begin{proposition}\label{prop:A_phi_closure_dom}
Let the notations and assumptions be as in Proposition~\ref{thm:m-accr_A_phi}. Then
$$
\overline{\dom(A_\phi)} = L^1(\R^d;\overline{I}).
$$
\end{proposition}

\subsubsection{The Cauchy problem}

Let $A$ be an $m$-accretive operator on a Banach space $X$ and consider the associated Cauchy problem
\begin{equation}\label{eq:abstract_Cauchy}
    \begin{cases}
        y'(t)+Ay(t) = f(t), & t \in [0,T],\\
        y(0) = y_0,
    \end{cases}
\end{equation}
for a given $T \in (0,\infty)$, where $y_0 \in X$ and $f \in L^1(0,T;X)$.

We recall the notion of  mild solutions for \eqref{eq:abstract_Cauchy} 
(see \cite[Definition~4.3]{Ba10}, cf.\ \cite[Definition~10.6]{Va07}).

\begin{definition}\label{def:mild_sol_Cauchy}
A \emph{mild solution} of the Cauchy problem \eqref{eq:abstract_Cauchy}  is a function $y \in C([0,T];X)$ with the property that for each $\varepsilon > 0$ there is an $\varepsilon$-approximate solution $z$ of $z'+Az=f$ on $[0,T]$ obtained via implicit time discretizations (in the sense of \cite[Definition~4.2]{Ba10} and \cite[Section~10.2]{Va07}) such that $\norm{y(t)-z(t)} \leq \varepsilon$ for all $t \in [0,T]$ and $y(0)=y_0$.  
\end{definition}

\begin{remark}\label{rmk:mild_solution_concept_strong}
For our problem we cannot consider the  concept of \emph{strong solutions} \cite[Definition~4.1]{Ba10}.
Strong solutions are mild solutions but the opposite is generally not true as simple examples in \cite[pg.~140,141]{Ba10} show. 
It is a natural question under what conditions mild solutions are strong solutions.
In fact, these solution concepts coincide under additional assumptions on the Banach space plus extra regularity assumptions on $f$ and $y_0$, see \cite[Theorem~4.5]{Ba10}.
Unfortunately, the extra assumption for the Banach space is reflexivity 
which is not satisfied by $L^1$ that is of interest in our case.
Properties of mild solutions can also be shown using the theory of subdifferentials, see Theorem \ref{thm:energy_est;Lip_source_term_f}.
\end{remark}

As $\varepsilon$-approximate solutions take their values in the domain of the operator $A$, it directly follows from the above definition that mild solutions take their values in the closure of the domain of $A$:
 
\begin{proposition}\label{prop:mild_sol_valued_closure_domain}
Let $y \in C([0,T];X)$ be a mild solution of the  Cauchy problem \eqref{eq:abstract_Cauchy}.
Then $y \in C([0,T];\overline{\dom(A)})$.
In particular, $y_0 = y(0) \in \overline{\dom(A)}$.
\end{proposition}

The following theorem provides existence and uniqueness of mild solutions for problem \eqref{eq:abstract_Cauchy}.
\begin{theorem}\label{thm:exist_mild_sol_Cauchy}
Let $A$ be an m-accretive operator on a Banach space $X$. Then, for each $y_0 \in \overline{\dom(A)}$ and $f \in L^1(0,T;X)$, there is a unique mild solution $y$ to \eqref{eq:abstract_Cauchy}.
Moreover, if $y$ and $\overline{y}$ are two mild solutions to \eqref{eq:abstract_Cauchy} corresponding to $f$, $y_0$ and $\overline{f}$, $\overline{y}_0$, respectively, then
\begin{equation}\label{eq:thm:exist_mild_sol_Cauchy;contr}
    \norm{y-\overline{y}}_{C([0,T];X)} \leq \norm{y_0-\overline{y}_0} + \norm{f-\overline{f}}_{L^1(0,T;X)}.
\end{equation}
\end{theorem}
\begin{proof}
The first part concerning the existence of a unique mild solution corresponds to \cite[Corollary~4.1]{Ba10} and the contraction property \eqref{eq:thm:exist_mild_sol_Cauchy;contr} follows from a combination of \cite[(4.14)]{Ba10} and \cite[Proposition~3.7(iv)]{Ba10}.  
\end{proof}

The following lemma is used to obtain compatibility of the  mild solutions corresponding to the $L^1(\R^d)$ setting and the $\dot{H}^{-1}(\R^d)$ setting. 

\begin{lemma}\label{lem:comp_mild_solutions}
Let $(X,Y)$ be a compatible couple of Banach spaces, let $A$ be an m-accretive operator on $X$ and let $B$ be an m-accretive operator on $Y$. 
Assume that $A$ and $B$ are resolvent compatible in the sense that, for each $\lambda > 0$ and $z \in X \cap Y$,
\begin{equation}\label{eq:lem:comp_mild_solutions}
(I+\lambda A)^{-1}z = (I+\lambda B)^{-1}z.    
\end{equation}
Let $z_0 \in \overline{\dom(A)} \cap \overline{\dom(B)}$, let $f \in L^1(0,T;X \cap Y)$ and let $x$ and $y$ be the unique mild solutions to
\begin{equation}\label{eq:abstract_Cauchy;A,X}
    \begin{cases}
        x'(t)+Ax(t) = f(t), & t \in [0,T],\\
        x(0) = z_0,
    \end{cases}
\end{equation}
and
\begin{equation}\label{eq:abstract_Cauchy;B,Y}
    \begin{cases}
        y'(t)+By(t) = f(t), & t \in [0,T],\\
        y(0) = z_0,
    \end{cases}
\end{equation}
respectively.
Then $x=y$.
\end{lemma}
\begin{proof}
A combination of \eqref{eq:lem:comp_mild_solutions} and the iterative scheme \cite[(10.14)]{Va07} yields that we can construct simultaneous $\varepsilon$-approximate solutions for  \eqref{eq:abstract_Cauchy;A,X} and \eqref{eq:abstract_Cauchy;B,Y}.
As both $x$ and $y$ are obtained by taking the limit $\varepsilon \to 0$, we find that $x=y$.
\end{proof}

Using the contraction property \eqref{eq:thm:exist_mild_sol_Cauchy;contr} in Theorem~\ref{thm:exist_mild_sol_Cauchy}, by means of a standard fixed-point argument we can treat the following abstract problem with a non-linear right hand side.

\begin{corollary}\label{cor:thm:exist_mild_sol_Cauchy;Lip_non}
Let $A$ be an m-accretive operator on a Banach space $X$ and let $F:\overline{\dom(A)} \to X$ be a Lipschitz function. 
Then, for each $y_0 \in \overline{\dom(A)}$, there is a unique mild solution $y \in C([0,T];\overline{\dom(A)})$ to
\begin{equation}\label{eq:cor:thm:exist_mild_sol_Cauchy;Lip_non}
\begin{cases}
        y'(t)+Ay(t) = F(y(t)), & t \in [0,T],\\
        y(0) = y_0.
\end{cases}    
\end{equation}    
\end{corollary}
\begin{proof}
By the continuation of mild solutions \cite[Proposition~10.12(ii)]{Va07}, it suffices to consider the case $T < L:= [F]_{\mrm{Lip}}$.  
Now fix $y_0 \in \overline{\dom(A)}$ and consider the complete metric space
$$
\F := \{ y \in C([0,T];\overline{\dom(A)}) : y(0)=y_0 \}.
$$
Let $\ms{G}:L^1(0,T;X) \to \F$ be the solution operator to \eqref{eq:abstract_Cauchy} obtained from Theorem~\ref{thm:exist_mild_sol_Cauchy} and consider the mapping $\ms{T}:\F \to \F$ defined by $\ms{T}(y) := \ms{S}(F(y))$.
Then, for $y, z \in \F$,
\begin{align*}
\nrm{\ms{T}(y)-\ms{T}(z)}_{\F}
&=\nrm{\ms{G}(F(y))-\ms{G}(F(z))}_{\F} \stackrel{\eqref{eq:thm:exist_mild_sol_Cauchy;contr}}{\leq} \nrm{F(y)-F(z)}_{L^1(0,T;X)}   \\
&\leq L\nrm{y-z}_{L^1(0,T;X)} \leq LT\nrm{y-z}_{C([0,T];X)}=LT\nrm{y-z}_{\F}.
\end{align*}
Since $LT<1$, it follows from Banach's fixed-point theorem that there is a unique $y \in \F$ such that $y=\ms{T}(y)$.
\end{proof}

\subsubsection{Subpotential maximal monotone operators}

Let $X$ be a Banach space. 
An operator $A:X \to X^*$ is said to be \emph{monotone} if, for all $(x_1,y_1), (x_2,y_2) \in A$,
$$
\ip{x_1-x_2,y_1-y_2} \geq 0.
$$
A monotone operator $A:X \to X^*$ is said to be \emph{maximal monotone} if it is not properly contained in any other monotone operator $X \to X^*$.

For $X=H$ a Hilbert space that is identified with its dual $X^*$, an operator $A:H \to H$ is accretive if and only if is monotone.
As a consequence, in this case, an operator $A:H \to H$ is $m$-accretive if and only if is maximal monotone.

We will consider maximal monotone operators that are subdifferentials of a lower semicontinuous (l.s.c.) convex function, also referred to as  \emph{subpotential maximal monotone operators}.
See Theorem~\ref{thm:subpotential_maximal_monotone_operator} below, which can for instance be found in \cite[Theorem~2.8]{Ba10}.

Before we state the theorem, let us provide the definition of subdifferentials.
Let $\varphi:X \to \overline{\R}$ be a l.s.c.\ convex proper function.  
We write
$$
\dom(\varphi) := \{ x \in X : \varphi(x) < \infty \}.
$$
The mapping $\p\varphi:X \to X^*$ defined by
$$
\p\varphi(x) = \{ x^* \in X^* : \varphi(x) \leq \varphi(y)+\ip{x^*,x-y}, \forall y \in X \}
$$
is called the \emph{subdifferential} of $\varphi$.
Note that, in general, $\p\varphi$ is multi-valued.
We write 
$$
\dom(\p\varphi) := \{ x \in X: \p\varphi(x) \neq \emptyset \}.
$$

\begin{theorem}\label{thm:subpotential_maximal_monotone_operator}
Let $X$ be a real Banach space and let $\varphi:X \to \overline{\R}$ be a l.s.c.\ proper convex function. Then $\p\varphi$ is a maximal monotone operator.    
\end{theorem}

For us the importance of subpotential maximal monotone operators comes from Brezis' maximal $L^2$-regularity theorem \cite{Br71}. 
We will use this result in the form of \cite[Theorem~2.2]{AH20} (cf.\ \cite[Theorem~4.11]{Ba10}) as an abstract way to use energy methods in Theorem~\ref{thm:energy_est;Lip_source_term_f}.

\subsubsection{Stability}\label{sect_prelim_stab}

One way to obtain extra information on the regularity of mild solutions is by approximating the equation.
To this end, we recall the notion of convergence of $m$-accretive operators.
Let $(A_n)_{n \in \N}$ be a sequence of $m$-accretive operators on a Banach space $X$ and let $A$ be an $m$-accretive operator on $X$. 
Following \cite[pg.~164]{BC81} (cf.\ \cite[Proposition~4.4]{Ba10}), we say that $A_n \to A$ as $m$-accretive operators in $X$ as $n \to \infty$ if
\begin{equation}\label{eq:conv_m-accr_op}
\lim_{n \to \infty}(I+A_n)^{-1}x = (I+A)^{-1}x, \qquad x \in X.    
\end{equation}

The following theorem can be considered the nonlinear version of the Trotter--Kato theorem from the theory of $C_0$-semigroups.
\begin{theorem}\label{thm:nonlin_Trotter-Kato}
Let $(A_n)_{n \in \N}$ be a sequence of $m$-accretive operators on a Banach space $X$ and let $A$ be an m-accretive operator on $X$.  
For each $n \in \N$, let $f^n \in L^1(0,T;X)$, $y^n_0 \in \overline{\dom(A_n)}$ and let $y_n$ be the mild solution to 
\begin{equation*}\label{eq:abstract_Cauchy_n}
    \begin{cases}
        y_n'(t)+A_ny_n(t) = f^n(t), & t \in [0,T],\\
        y_n(0) = y^n_0.
    \end{cases}
\end{equation*}
If $A_n \to A$ as $m$-accretive operators in $X$ as $n \to \infty$, $f=\lim_{n \to \infty} f^n$ in $L^1(0,T;X)$ and $y_0=\lim_{n \to \infty}y^n_0$ in $X$, then $y_n \to y$ in $C([0,T];X)$ as $n \to \infty$, where $y$ is the mild solution to \eqref{eq:abstract_Cauchy}.
\end{theorem}
\begin{proof}
Inspection of the proof of \cite[Proposition~4.4]{Ba10} shows that the 'if'-part of the statement only uses \cite[(4.105)]{Ba10} for an arbitrary $\lambda$, where $\lambda$ is a resolvent parameter from the condition \cite[(4.105)]{Ba10}. 
In particular, we may take $\lambda = 1 \in (0,\infty)$. 
Therefore, \cite[Theorem~4.14]{Ba10} is applicable and yields the desired result.   
\end{proof}


\section{Well-posedness}\label{sec:well}

In this section we prove Theorem \ref{thm:wellposedness;main1} except for the last statement, i.e. we show the well-posedness  for Problem 
\eqref{eq:GPME}. The non-negativity of solutions is shown in Section \ref{sec:approx}.
To show well-posedness we formulate the problem as an abstract Cauchy problem in $L^1(\R^d)$ and apply the theory of $m$-accretive operators. 
To show that solutions take values in the open interval $(-1,1)$ requires more work. We need to consider the Cauchy problem in $\dot H^{-1}(\R^d)$ and provide a suitable characterization of the corresponding elliptic operator.

In the following two theorems we consider  \eqref{eq:GPME} as an abstract Cauchy problem of the form \eqref{eq:cor:thm:exist_mild_sol_Cauchy;Lip_non} in the Banach space $X=L^1(\R^d)$ for the operator  $A=A_\phi$ in Proposition~\ref{thm:m-accr_A_phi}.
The corresponding notion of mild solution was specified in  Definition~\ref{def:mild_sol_Cauchy}.

\begin{theorem}\label{thm:wellposedness;Lip_source_term_f}
Let $I$ be an open interval with $0\in I$, $\phi \colon I \to \R$ be a maximal monotone function with $\phi(0)=0$ and let $f\colon \overline{I} \to \R$ be a Lipschitz function with $f(0)=0$. Then for every $u_0\in L^1(\R^d;\overline{I})$ there exists a unique mild solution $u \in C([0,T];L^1(\R^d;\overline{I}))$ of
\begin{equation}\label{eq:GPME2}
\begin{cases}
\p_t u = \Delta \phi(u)+f(u) &\text{in}\:\:(0,T) \times \R^d, \\
u(0)=u_{0} &\text{on}\:\:\R^d.    
\end{cases}
\end{equation}
\end{theorem}

In Corollary~\ref{cor:thm:energy_est;Lip_source_term_f;[-1,1]} we will even see that if $I$ is bounded then for all $t \in (0,T]$, the solution $u(t,x)$ takes its values in the open interval $I$ for almost all $x \in \R^d$.

\begin{proof}[Proof of Theorem~\ref{thm:wellposedness;Lip_source_term_f}]
Note that
$$
F:L^{1}(\R^d;\overline{I}) \to L^1(\R^d),\, v \mapsto f(v),
$$
is a well-defined Lipschitz mapping with $[F]_{\mrm{Lip}} \leq L$.
Indeed, we observe that $f$ satisfies $|f(z)| \leq L|z|$ for all $z \in \overline{I}$, thanks to $f(0)=0$ and $L=[f]_{\mrm{Lip}}$. 
Combining Proposition~\ref{thm:m-accr_A_phi}, Proposition~\ref{prop:A_phi_closure_dom} and Corollary~\ref{cor:thm:exist_mild_sol_Cauchy;Lip_non} we obtain that for every $u_0\in L^1(\R^d;\overline{I})$ there exists a unique mild solution $u \in C([0,T];L^1(\R^d;\overline{I}))$ of \eqref{eq:GPME}.
\end{proof}

The main interest of this paper is the Sobolev regularity of mild solutions given by Theorem \ref{thm:wellposedness;Lip_source_term_f}.
In that context the term $f(u)$ in \eqref{eq:GPME} plays a minor role as the proof of Theorem \ref{thm:wellposedness;Lip_source_term_f} shows that 
$f(u)\in L^1((0,T)\times\R^d)$. Hence, $f(u)$ can be considered as a given function and it suffices to consider the following situation.

\begin{theorem}\label{thm:wellposedness}
Let $\phi$ satisfy the assumptions in Theorem \ref{thm:wellposedness;Lip_source_term_f} and  
$f\in L^1((0,T)\times\R^d)$.
Then, for every $u_0\in L^1(\R^d;\overline{I})$ there exists a unique mild solution $u\in C([0,T];L^1(\R^d;\overline{I}))$ of the Cauchy problem 
\begin{equation}\label{eq:GPME_f_lin}
\begin{cases}
\p_t u = \Delta \phi(u)+f &\text{in}\:\:(0,T) \times \R^d, \\
u(0)=u_{0} &\text{on}\:\:\R^d.    
\end{cases}
\end{equation}
\end{theorem}

\begin{proof}
This follows from a combination of Proposition~\ref{thm:m-accr_A_phi}, Proposition~\ref{prop:A_phi_closure_dom} and Theorem~\ref{thm:exist_mild_sol_Cauchy}.
\end{proof}

\begin{remark}\label{rmk:Lp_all_p}
For every $p \in [1,\infty]$, there is the continuous embedding
\begin{equation}\label{eq:rmk:Lp_all_p;incl}
    L^1(\R^d;[-1,1]) \hookrightarrow L^p(\R^d).
\end{equation}
Indeed, for $v \in L^1(\R^d;[-1,1])$ we have 
$|v| \leq 1$ and thus
\begin{equation}\label{eq:rmk:Lp_all_p;est}
\nrm{v}_{L^p(\R^d)} = \left( \int_{\R^d} |v(x)|^{p} \right)^{1/p} \leq
\left( \int_{\R^d} |v(x)| \right)^{1/p} = \nrm{v}_{L^1(\R^d)}^{\frac{1}{p}}.  
\end{equation}
\end{remark}

Remark~\ref{rmk:Lp_all_p} provides us with the flexibility of having the full range of $L^p$-spaces at our disposal.
In particular, if the dimension $d \geq 3$, we can choose $p \in (1,2)$ such that $\frac{1}{p}=\frac{1}{2}+\frac{1}{d}$. 
Then there is the embedding \eqref{eq:Sob_embd_H^-1}.
This observation will allow us to consider \eqref{eq:GPME} in $\dot H^{-1}(\R^d)$.

Let $\phi$ be as in Theorem \ref{thm:wellposedness;Lip_source_term_f}. We denote by $\Phi$ the primitive of $\phi$ with $\Phi(0)=0$, that is 
\begin{equation}
    \Phi(r) := \int_{0}^{r}\phi(z)\,\d z, \qquad r \in I.
\end{equation}
Note that $\Phi \geq 0$ since $\phi(0)=0$ and that $\phi$ is monotonically increasing.
If $I\neq \R$, it will be convenient to view $\phi$ and $\Phi$ defined on $\R$ by extending $\phi$ as $\pm\infty$ on $\R\setminus I$ and  $\Phi$ as $\infty$ on $\R \setminus I$. 

For a measurable function $v \colon \R^d \to \R$ we define the \textit{energy} $\varphi(v) \in [0,\infty]$ by
\begin{equation}\label{eq:def_varphi}
\varphi(v) :=  \int_{\R^d}\Phi(v)\,\d x = \nrm{\Phi(v)}_{L^1(\R^d)}.    
\end{equation}

\begin{theorem}\label{thm:energy_est;Lip_source_term_f}
Let the notation and assumptions be as in Theorem~\ref{thm:wellposedness;Lip_source_term_f} and $I=(-1,1)$.
In addition, assume that $d \geq 3$, let $p \in (1,2)$ be given by $\frac{1}{p}=\frac{1}{2}+\frac{1}{d}$ and assume that there exists $q \in [p,\infty)$ such that 
\begin{equation}\label{eq:thm:energy_est;Lip_source_term_f;Phi_Lq}
    \left\{ v \in L^0(\R^d) : \Phi(v) \in L^1(\R^d) \right\} \subset L^q(\R^d) \qquad \text{boundedly},
\end{equation}
i.e. for every $R>0$ there exists $M(R)>0$ such that for all $v \in L^0(\R^d)$ with $\nrm{\Phi(v)}_{L^1(\R^d)} \leq R$ it holds that $\nrm{v}_{L^q(\R^d)} \leq M(R)$.

Then the mild solution $u$ satisfies $\phi(u) \in L^{2}_\loc((0,T];\dot{H}^1(\R^d))$,
\begin{equation}\label{eq:thm:energy_est;Lip_source_term_f}
\p_t u = \Delta \phi(u)+f(u)  \qquad\text{in}\:\:(0,T) \times \R^d  
\end{equation}
in the sense of distributions and $\varphi(u) \in W^{1,1}_\loc((0,T]) \cap L^1(0,T)$.
Moreover, the following estimates hold, 
\begin{equation}
\int_0^T\varphi(u(s))\,\d s \leq C_d^2 T(1+L^2)\nrm{u}_{C([0,T];L^1(\R^d))}^{1+\frac{2}{d}}+C_d^2\nrm{u_0}_{L^1(\R^d)}^{1+\frac{2}{d}},
\end{equation}
\begin{equation}
t\varphi(u(t)) \leq \int_0^T\varphi(u(s))\,\d s + C_d^2\frac{T^{2}}{2}\nrm{u}_{C([0,T];L^1(\R^d))}^{1+\frac{2}{d}},    
\end{equation}
and
\begin{equation}
\nrm{t \mapsto \sqrt{t}\,\nabla \phi(u(t))}_{L^2((0,T)\times\R^d)}^2 \leq 4\int_0^T\varphi(u(s))\,\d s + 2C_d^2 T^{2}\nrm{u}_{C([0,T];L^1(\R^d))}^{1+\frac{2}{d}},    
\end{equation}
where $L=[f]_{\mrm{Lip}}$.

If, in addition, $\Phi(u_0) \in L^1(\R^d)$, then $\phi(u) \in L^2(0,T;\dot{H}^1(\R^d))$ and $\varphi(u) \in W^{1,1}([0,T])$, and we have
\begin{equation}
\frac{1}{2}\nrm{\nabla \phi(u)}_{L^2((0,T)\times\R^d)}^2   + \nrm{\varphi(u)}_{L^\infty(0,T)} \leq \frac{1}{2}C_d^2 TL^2\nrm{u}_{C([0,T];L^1(\R^d))}^{1+\frac{2}{d}} + \varphi(u_0). 
\end{equation}    
\end{theorem}

An immediate consequence of Theorem \ref{thm:energy_est;Lip_source_term_f} is that the mild solution takes values in the open interval $(-1,1).$

\begin{corollary}\label{cor:thm:energy_est;Lip_source_term_f;[-1,1]}
Let the assumptions in Theorem~\ref{thm:energy_est;Lip_source_term_f} be satisfied.
Then, for all $t \in (0,T]$ we have $\Phi(u(t)) \in L^1(\R^d)$ and thus $u(t,x) \in (-1,1)$ for almost every $x \in \R^d$.     
\end{corollary}

\begin{example}\label{ex:eq:thm:energy_est;Lip_source_term_f;Phi_Lq}
Let $\phi$ be as in Example~\ref{ex:max_mon_fc_biofilm}. 
Then the condition \eqref{eq:thm:energy_est;Lip_source_term_f;Phi_Lq} is satisfied for all $q \in [b+2,\infty]$.   
\end{example}

For the proof of Theorem \ref{thm:energy_est;Lip_source_term_f} we need the following theorem that provides a characterization of the elliptic operator in \eqref{eq:thm:energy_est;Lip_source_term_f} in $\dot{H}^{-1}(\R^d)$.
Below we view the energy function $\varphi$ defined in \eqref{eq:def_varphi} as a mapping $\dot{H}^{-1}(\R^d) \to [0,\infty]$ by extending it as
\begin{equation}
    \varphi(u) := \begin{cases}
    \int_{\R^d}\Phi(v)\,\d x, & v \in \dot{H}^{-1}(\R^d) \cap L^q(\R^d),\\
    \infty, & \text{otherwise},
    \end{cases}
\end{equation}
where $q$ is as in Theorem \ref{thm:energy_est;Lip_source_term_f}.

\begin{theorem}\label{thm:subdiff_GPME_whole_space}
Let $I\subset \R$ be an open interval with $0\in I$ and $\phi \colon I \to \R$ be a maximal monotone function with $\phi(0)=0$.
Assume that $d \geq 3$, let $p$ and $q$ be as in Theorem \ref{thm:energy_est;Lip_source_term_f} and assume that \eqref{eq:thm:energy_est;Lip_source_term_f;Phi_Lq} holds.
Then $\varphi\colon\dot{H}^{-1}(\R^d) \to [0,\infty]$ is a proper convex, lower semicontinuous function.
Moreover, its subdifferential $\p \varphi \subset \dot{H}^{-1}(\R^d) \times \dot{H}^{-1}(\R^d)$ is given by
\begin{align}
    \p \varphi &= \{(u,-\Delta \phi(u)): u \in \dot{H}^{-1}(\R^d) \cap L^q(\R^d), \phi(u) \in \dot{H}^{1}(\R^d) \}.    \label{eq:thm:subdiff_GPME_whole_space} 
\end{align}
\end{theorem}

\begin{remark}\label{rmk:thm:subdiff_GPME_whole_space}
In Theorem~\ref{thm:subdiff_GPME_whole_space} we identify the Hilbert space $\dot{H}^{-1}(\R^d)$ with its dual $\dot{H}^{1}(\R^d)$ through the isomorphism $-\Delta:\dot{H}^{1}(\R^d) \to \dot{H}^{-1}(\R^d)$.
Without this identification, the last statement becomes that the
subdifferential $\p \varphi \subset \dot{H}^{-1}(\R^d) \times \dot{H}^{1}(\R^d)$ is given by
\begin{align}
    \p \varphi &= \{(u,\phi(u)): u \in \dot{H}^{-1}(\R^d) \cap L^q(\R^d), \phi(u) \in \dot{H}^{1}(\R^d) \}.    \label{eq:rmk:thm:subdiff_GPME_whole_space} 
\end{align}
\end{remark}

For bounded domains with Dirichlet boundary conditions Theorem \ref{thm:subdiff_GPME_whole_space} is a variant of a classical result due to Brezis \cite[Theorem~17]{Br71b}, see e.g.\ \cite[Proposition~2.10]{Ba10} or \cite[Proposition~10.8]{Va07}.
The main difficulty in $\R^d$ compared to a bounded domain $\Omega \subset \R^d$ comes from the fact that the usual Sobolev spaces $H^1_0(\Omega)$ and $H^{-1}(\Omega)$ have to be replaced by the homogeneous Sobolev spaces $\dot H^1(\R^d)$ and $\dot H^{-1}(\R^d)$, respectively. 
Furthermore, the argument in the beginning of the proof of \cite[Proposition~2.10]{Ba10} that is based on the assumption \cite[(2.8)]{Ba10} breaks down as the Lebesgue measure of $\R^d$ is not finite. We overcome this by replacing that assumption by \eqref{eq:thm:energy_est;Lip_source_term_f;Phi_Lq}.

For the proof of Theorem~\ref{thm:subdiff_GPME_whole_space} we will need the following technical lemma (cf.\ \cite[Lemma~2.6]{Ba10}).

\begin{lemma}\label{lem:thm:subdiff_GPME_whole_space}
Suppose that $d \geq 3$. Let $p \in (1,2)$ satisfy $\frac{1}{p}=\frac{1}{2}+\frac{1}{d}$ and let $q \in (p,\infty)$.    
Let $f \in \dot{H}^{-1}(\R^d) \cap L^q(\R^d)$ and $g \in \dot{H}^{1}(\R^d)$.
If $fg \geq 0$ a.e., then $fg \in L^1(\R^d)$ and
$$
\ipb{f,g} = \int_{\R^d}fg \geq 0,
$$
where $\ipb{\cdot,\cdot}$ denotes the duality pairing between $\dot{H}^{-1}(\R^d)$ and $\dot{H}^{1}(\R^d)$.
\end{lemma}
\begin{proof}
Pick $\chi \in C^\infty_c(\R^d)$ with $\chi \geq 0$ and $\chi \equiv 1$ in a neighborhood of $0$ and such that $\chi$ is radially decreasing. 
Let $\chi_n(x):=\chi(\frac{x}{n})$ and 
$g_n := \chi_n g$. We observe that 
$$
\nabla g_n =(\nabla\chi_n)g + \chi_n \nabla g= 
\underbrace{\frac{1}{n}\left(\nabla\chi\left(\frac{\cdot}{n}\right) \right)g}_{=:u_n} + \underbrace{\chi_n \nabla g}_{=:v_n}.
$$
By the Lebesgue dominated convergence theorem, we have $v_n \to \nabla g$ in $L^2(\R^d)$ as $n \to \infty$. Let us next treat $u_n$.
As 
$$
\frac{1}{2} = \frac{1}{p'}+\frac{1}{p}-\frac{1}{2} = \frac{1}{p'} + \frac{1}{d},
$$
we have by H\"older's inequality and a substitution of variables 
\begin{align*}
\nrm{u_n}_{L^2(\R^d)} 
&\leq \norm{\frac{1}{n}\nabla\chi\left(\frac{\cdot}{n}\right)}_{L^d(\R^d)}\nrm{g}_{L^{p'}(\R^d)} = \norm{\nabla\chi}_{L^d(\R^d)}\nrm{g}_{L^{p'}(\R^d)}.
\end{align*}
This shows that the sequence $(u_n)_{n\in\N}$ is bounded in $L^2(\R^d)$.
Therefore, it has a weakly convergent subsequence,
say $u_{n_k} \rightharpoonup u$ in $L^2(\R^d)$ as $k \to \infty$.
Testing against $\eta \in \ms{D}(\R^d)$, we find that $u=0$ because
$$
\supp(u_{n_k}) \subset \supp\left(\nabla\chi\left(\frac{\cdot}{n}\right)\right) \subset \R^d \setminus B(0,n).
$$

Combining the above, we obtain that $\nabla g_{n_k} \rightharpoonup \nabla g$ in $L^2(\R^d)$ as $k \to \infty$.
This means that $g_{n_k} \rightharpoonup g$ in $\dot{H}^1(\R^d)$ as $k \to \infty$.
Since $g_{n_k}  \in H^1_c(\R^d) \subset L^{p'}_c(\R^d) \subset L^{q'}(\R^d)$ by \eqref{eq:Sob_embd_H^-1;dual} and $p' > q'$, it follows that
\begin{align*}
    \ipb{f,g} = \lim_{k \to \infty} \ipb{f,g_{n_k}} = \lim_{k \to \infty} \int_{\R^d}f g_{n_{k}}.
\end{align*}
By construction, $0 \leq f g_{n_k} \nearrow fg$ a.e.\ as $k \to \infty$. 
Therefore, by the Beppo Levi theorem, 
$$
\lim_{k \to \infty} \int_{\R^d}f g_{n_{k}} = \int_{\R^d}fg.
$$
The desired result now follows.
\end{proof}

\begin{proof}[Proof of Theorem~\ref{thm:subdiff_GPME_whole_space}]
First, we show that $\varphi$ is a proper convex function.
To see that it is lower semicontinuous, let $\lambda > 0$ and let $(u_{n})_{n\in\N} \subset \dot{H}^{-1}(\R^d)$ be such that $u_n \to u$ in $\dot{H}^{-1}(\R^d)$ 
and $\varphi(u_{n}) \leq \lambda$. 
Then $(u_{n})_{n\in\N} \subset \dot{H}^{-1}(\R^d) \cap L^q(\R^d)$ and, thanks to the assumption \eqref{eq:thm:energy_est;Lip_source_term_f;Phi_Lq},
$$
\nrm{u_n}_{L^q(\R^d)} \leq M(\lambda), \qquad n \in \N.
$$
As a consequence, $(u_n)_{n\in\N}$ has a weakly convergent subsequence $(u_{n_k})_{k\in\N}$ in $L^q(\R^d)$. As weak convergence in $L^q(\R^d)$ implies convergence in $\dot{\ms{S}}'(\R^d)$ and $\dot{H}^{-1}(\R^d) \hra \dot{\ms{S}}'(\R^d)$,
where $\dot{\ms{S}}'(\R^d)$ is the space of tempered distributions modulo polynomials,
we find that $u$ is the weak limit of $(u_{n_k})_{k\in\N}$ in $L^q(\R^d)$.  
By \cite[Proposition~2.7]{Ba10},
$$
\varphi_q : L^q(\R^d) \to [0,\infty],\: v \mapsto \int_{\R^d}\Phi(v),
$$
is a proper convex, lower semicontinuous function, and thus as a consequence of Mazur's theorem (see \cite[page~5]{Ba10}) also  weakly lower semicontinous.
Therefore,
$$
\varphi(u)=\varphi_q(u) \leq \liminf_{n \to \infty}\varphi_q(u_n) = \liminf_{n \to \infty}\varphi(u_n) \leq \lambda. 
$$
This shows that the level set $\{ u \in \dot{H}^{-1}(\R^d) : \varphi(u) \leq \lambda \}$ is closed for each $\lambda > 0$, which means that $\varphi$ is lower semicontinuous.

Next, we establish the final statement \eqref{eq:thm:subdiff_GPME_whole_space} in the form \eqref{eq:rmk:thm:subdiff_GPME_whole_space} following Remark~\ref{rmk:thm:subdiff_GPME_whole_space}.
Denoting the operator on the right-hand side of \eqref{eq:rmk:thm:subdiff_GPME_whole_space} by $A$, it suffices to show that $A$ is monotone and that  $\p \varphi \subset A$.
Indeed, as $\p \varphi$ is maximal monotone by Theorem~\ref{thm:subpotential_maximal_monotone_operator}, this implies that $\p \varphi = A$.

In order to show that $A$ is monotone, let $(u,\phi(u)), (v,\phi(v)) \in A$. 
Then $(u-v)(\phi(u)-\phi(v)) \geq 0$ as $\phi$ is monotonically increasing.
By Lemma~\ref{lem:thm:subdiff_GPME_whole_space} we thus obtain that
$$
\ipb{u-v,\phi(u)-\phi(v)} \geq 0,
$$
which proves that $A$ is  monotone.

Finally, let us show that $\p \varphi \subset A$.
To this end, fix $(u,w) \in \p \varphi$, that is, $u \in \dom(\p \varphi)$ and $w \in \p \varphi(u)$.
By \cite[Proposition~1.6]{Ba10} it holds that $\dom(\p \varphi) \subset \dom(\varphi) = \dot{H}^{-1}(\R^d) \cap L^q(\R^d)$. 
In particular, we have $u \in \dot{H}^{-1}(\R^d) \cap L^q(\R^d)$.
Furthermore, by definition of $\p \varphi(u)$, we have
\begin{equation}\label{eq:ineq_varphi}
    \varphi(u)-\varphi(v) \leq \ipb{u-v,w}
\end{equation}
for all $v \in \dom(\varphi) = \dot{H}^{-1}(\R^d) \cap L^q(\R^d)$.

Let $B=B(x_0,r)$ for some $x_0 \in \R^d$ and $r>0$ and let $u_0 \in I$.
Then $u_0\one_{B}, u\one_{B} \in L^q_c(\R^d) \subset L^{p}(\R^d)$ as $q \geq p$, so that $u_0\one_{B}, u\one_{B} \in \dot{H}^{-1}(\R^d) \cap L^q(\R^d)$ by \eqref{eq:Sob_embd_H^-1}.
Defining
$$
v := u_0\one_B + u\one_{\R^d \setminus B} = u + (u_0-u)\one_{B} \in \dot{H}^{-1}(\R^d) \cap L^q(\R^d),
$$
we find that
\begin{align*}
\varphi(u\one_B)-\varphi(u_0\one_B)
&= \varphi(u\one_B)+\varphi(u\one_{\R^d \setminus B}) -(\varphi(u_0\one_B)+\varphi(u\one_{\R^d \setminus B})) \\
&= \varphi(u)-\varphi(v) \stackrel{\eqref{eq:ineq_varphi}}{\leq} \ipb{u-v,w} = \ipb{(u-u_0)\one_B,w}.
\end{align*}
As $(u-u_0)\one_B \in L^p(\R^d)$ and $w \in \dot{H}^1(\R^d) \subset L^{p'}(\R^d)$ by \eqref{eq:Sob_embd_H^-1;dual},
we have $\ipb{(u-u_0)\one_B,w} = \int_{B}(u-u_0)w$. 
The above inequality can thus be rewritten as
$$
\int_{B}(\Phi(u)-\Phi(u_0)) \leq \int_{B}(u-u_0)w.
$$
As this holds for arbitrary $B=B(x_0,r)$ with $x_0 \in \R^d$ and $r>0$,
it follows that
$$
\Phi(u) \leq \Phi(u_0)+(u-u_0)w
$$
for all $u_0 \in I$.
The latter means that $w(x) \in \p\Phi(u(x))$ for a.e.\ $x \in \R^d$. 
As $\Phi$ is a convex function with a.e.\ derivative $\phi$, we have $\p\Phi(u(x))=\phi(u(x))$ for a.a.\ $x \in \R^d$ (see e.g. \cite[Example 3, page~8]{Ba10}).
Therefore,\ $w=\phi(u)$ and we conclude that $(u,w) \in A$.
\end{proof}

\begin{lemma}\label{lem:compatability_L1_subdifferential}
Let the assumptions and notation be as in Theorem~\ref{thm:subdiff_GPME_whole_space}.
Moreover, let $A_\phi$ be the operator on $L^1(\R^d)$ from Proposition~\ref{thm:m-accr_A_phi}.
Then $\p\varphi$ and $A_\phi$ are resolvent compatible in the sense that, for each $\lambda > 0$ and $f \in \dot{H}^{-1}(\R^d) \cap L^1(\R^d)$,
\begin{equation}\label{eq:lem:compatability_L1_subdifferential}
(I+\lambda\p\varphi)^{-1}f = (I+\lambda A_\phi)^{-1}f.    
\end{equation}
\end{lemma}

\begin{proof}
Fix $\lambda > 0$ and $f \in \dot{H}^{-1}(\R^d) \cap L^1(\R^d)$ and set $u:=(I+\lambda\p\varphi)^{-1}f$.
Then $u \in \dom(\p\varphi)$ so that $\phi(u) \in \dot{H}^1(\R^d) \subset L^{p'}(\R^d) \subset L^1_\loc(\R^d)$ by Theorem~\ref{thm:subdiff_GPME_whole_space} and \eqref{eq:Sob_embd_H^-1;dual}.
Moreover, 
$$
-\Delta \phi(u) = \lambda^{-1}(f-u).
$$
It suffices to show that $u \in L^1(\R^d)$.
Indeed, then $u \in \dom(A_\phi)$ with 
$$
(I+\lambda A_\phi)u = (I+\lambda \p\varphi)u = f,
$$
from which it follows that \eqref{eq:lem:compatability_L1_subdifferential} holds true.

In order to show that $u \in L^1(\R^d)$, we will formally multiply the equation
\begin{equation}\label{eq:lem:compatability_L1_subdifferential:EQ}
u-\lambda \Delta\phi(u) = f \qquad \text{in} \quad \dot{H}^{-1}(\R^d)    
\end{equation}
by $\sign_0(\phi(u))$ and integrate over $\R^d$.
In the spirit of \cite[page~112]{Ba10}, to make this rigorous, 
let $\gamma_\varepsilon:\R \to \R, \varepsilon > 0,$ be an approximation of 
$\sign_0$ as in \cite[(3.26)]{Ba10}. The explicit form is not important here, we will only need that $\gamma_\varepsilon$ is a monotonically increasing $BC^1$-function with $\gamma_\varepsilon(0)=0$ and such that $\gamma_\varepsilon \to \sign_0$ on $\R \setminus \{0\}$ as $\varepsilon \to 0$.

Note that $\nabla \gamma_\varepsilon(\phi(u)) = \gamma'_\varepsilon(u)\nabla \phi(u) \in L^2(\R^d)$, so $\gamma_\varepsilon(\phi(u)) \in \dot{H}^1(\R^d)$.
We can thus test \eqref{eq:lem:compatability_L1_subdifferential:EQ} against $\gamma_\varepsilon(\phi(u))$.
This yields
$$
\ip{u,\gamma_\varepsilon(\phi(u))}+\lambda\ip{ -\Delta\phi(u),\gamma_\varepsilon(\phi(u))} = \ip{f,\gamma_\varepsilon(\phi(u))}.
$$
Since
\begin{align*}
\ip{ -\Delta\phi(u),\gamma_\varepsilon(\phi(u))} 
&=
\ip{\nabla\phi(u),\nabla\gamma_\varepsilon(\phi(u))}  \\ 
&=
\ip{\nabla\phi(u),\gamma'_\varepsilon(u)\nabla \phi(u)} = \int_{\R^d}\gamma'_\varepsilon(u)|\nabla \phi(u)|^2 \geq 0
\end{align*}
and
$$
\ip{f,\gamma_\varepsilon(\phi(u))} \leq \norm{f}_{L^1(\R^d)}\norm{\gamma_\varepsilon(\phi(u))}_{L^\infty(\R^d)} \leq \norm{f}_{L^1(\R^d)},
$$
it follows that
$$
\ip{u,\gamma_\varepsilon(\phi(u))} \leq \norm{f}_{L^1(\R^d)}.
$$
As $u \in \dom(\p\varphi)$, we have $u \in \dot{H}^{-1}(\R^d) \cap L^q(\R^d)$ by Theorem~\ref{thm:subdiff_GPME_whole_space}.
Furthermore, as $\phi$ and $\gamma_\varepsilon$ are monotonically increasing functions with $\phi(0)=0$ and $\gamma_\varepsilon(0) = 0$, we have that $u$ and $\gamma_\varepsilon(\phi(u))$ have the same sign, so that $u\gamma_\varepsilon(\phi(u)) \geq 0$. 
We can thus invoke Lemma~\ref{lem:thm:subdiff_GPME_whole_space} to find that
$$
\int_{\R^d}u\gamma_\varepsilon(\phi(u)) = \ip{u,\gamma_\varepsilon(\phi(u))} \leq \norm{f}_{L^1(\R^d)}.
$$
Since $0 \leq u\gamma_\varepsilon(\phi(u)) \to u\sign_0(\phi(u))=u\sign_0(u)=|u|$ as $\varepsilon \to 0$, it follows by Fatou's lemma that 
$$
\norm{u}_{L^1(\R^d)} = \int_{\R^d}|u| \leq \liminf_{\varepsilon \to 0}\int_{\R^d}u\gamma_\varepsilon(\phi(u)) \leq \norm{f}_{L^1(\R^d)}.
$$
\end{proof}

\begin{lemma}\label{lem:IV_varphi}
Let the assumptions and notation be as in Theorem~\ref{thm:subdiff_GPME_whole_space} and $I=(-1,1)$.
Then $L^1(\R^d;[-1,1]) \subset \overline{\dom(\varphi)}$, where the closure is taken in $\dot{H}^{-1}(\R^d)$.
\end{lemma}

\begin{proof}
In light of Remark~\ref{rmk:Lp_all_p} and  the embedding \eqref{eq:Sob_embd_H^-1}, it 
suffices to show that $L^1(\R^d;[-1,1])$ is contained in the closure of $\dom(\varphi) \cap L^p(\R^d)$ in $L^p(\R^d)$.

To this end, let $f \in L^1(\R^d;[-1,1])$. 
Then there exists a sequence $(f_{n})_{n\in\N}$ of simple functions such that $|f_{n}| \leq 1-\frac{1}{n}$ and $f=\lim_{n \to \infty}f_n$ in $L^p(\R^d)$.
The observation that $f_n \in \dom(\varphi) \cap L^p(\R^d)$ finishes the proof.
\end{proof}

\begin{proof}[Proof of Theorem~\ref{thm:energy_est;Lip_source_term_f}]
Denoting by $C_d$ the norm of the embedding \eqref{eq:Sob_embd_H^-1}, this embedding combined with \eqref{eq:rmk:Lp_all_p;incl} in Remark~\ref{rmk:Lp_all_p} gives 
\begin{equation*}
\nrm{v}_{\dot{H}^{-1}(\R^d)} \leq C_{d}\nrm{v}_{L^1(\R^d;[-1,1])}^{\frac{1}{2}+\frac{1}{d}}, \qquad v \in L^1(\R^d;[-1,1]).    
\end{equation*}
Applying H\"older's inequality and the estimate above we conclude that for every $v \in C([0,T];L^1(\R^d;[-1,1]))$ we have
\begin{equation}\label{eq:thm:energy_est;Lip_source_term_f;proof_eq1}
\nrm{v}_{L^2(0,T;\dot{H}^{-1}(\R^d))}^2 \leq C_d^2T\nrm{v}^{1+\frac{2}{d}}_{C([0,T];L^1(\R^d;[-1,1]))} 
\end{equation}
and
\begin{equation}\label{eq:thm:energy_est;Lip_source_term_f;proof_eq2}
\nrm{t \mapsto \sqrt{t}\,v(t)}_{L^2(0,T;\dot{H}^{-1}(\R^d))}^2 \leq C_d^2\frac{T^2}{2}\nrm{v}^{1+\frac{2}{d}}_{C([0,T];L^1(\R^d;[-1,1]))}. 
\end{equation}
Furthermore, using the Lipschitz continuity of $f$,
$$
\nrm{f(v)}_{L^p(\R^d)} \leq L\nrm{v}_{L^p(\R^d)}, \qquad v \in L^p(\R^d;[-1,1]),
$$
we conclude that for every $v \in C([0,T];L^1(\R^d;[-1,1]))$ the following norm estimates hold, 
\begin{equation}\label{eq:thm:energy_est;Lip_source_term_f;proof_eq3}
\nrm{f(v)}_{L^2(0,T;\dot{H}^{-1}(\R^d))}^2 \leq C_d^2L^2T\nrm{v}^{1+\frac{2}{d}}_{C([0,T];L^1(\R^d;[-1,1]))} 
\end{equation}
and
\begin{equation}\label{eq:thm:energy_est;Lip_source_term_f;proof_eq4}
\nrm{t \mapsto \sqrt{t}\,f(v(t))}_{L^2(0,T;\dot{H}^{-1}(\R^d))}^2 \leq C_d^2L^2\frac{T^2}{2}\nrm{v}^{1+\frac{2}{d}}_{C([0,T];L^1(\R^d;[-1,1]))}. 
\end{equation}

Applying \eqref{eq:thm:energy_est;Lip_source_term_f;proof_eq3} to $v=u$, our mild solution from Theorem~\ref{thm:wellposedness;Lip_source_term_f}, we find that $f(u) \in L^2(0,T;\dot{H}^{-1}(\R^d))$. Furthermore, we have $u_0 \in \overline{\dom(\varphi)}$ by Lemma~\ref{lem:IV_varphi}.
Therefore, by \cite[Theorem~2.2]{AH20} (cf.\ \cite[Theorem~4.11]{Ba10}) and Theorem~\ref{thm:subdiff_GPME_whole_space}, there exists a unique solution $y \in H^1_\loc(0,T;\dot{H}^{-1}(\R^d)) \cap C([0,T];\dot{H}^{-1}(\R^d))$ of 
\begin{equation}\label{eq:GPME_u_y}
\begin{cases}
\p_t y = \Delta \phi(y)+f(u) &\text{in}\:\:(0,T) \times \R^d, \\
y(0)=u_{0} &\text{on}\:\:\R^d.    
\end{cases}
\end{equation}
This strong solution $y$ is a mild solution in $\dot{H}^{-1}(\R^d)$ for the operator $\p\varphi$ (see \cite[page~130]{Ba10}).
As $\p\varphi$ is a maximal monotone operator by Theorem~\ref{thm:subpotential_maximal_monotone_operator} and every maximal monotone operator on a Hilbert space is an $m$-accretive operator, combining Theorem~\ref{thm:subdiff_GPME_whole_space}, Lemma~\ref{lem:comp_mild_solutions} and
Lemma~\ref{lem:compatability_L1_subdifferential} implies that $y=u$.
All the statements now follow from \cite[Theorem~2.2]{AH20} and the estimates
\eqref{eq:thm:energy_est;Lip_source_term_f;proof_eq1}, \eqref{eq:thm:energy_est;Lip_source_term_f;proof_eq2}, \eqref{eq:thm:energy_est;Lip_source_term_f;proof_eq3}, \eqref{eq:thm:energy_est;Lip_source_term_f;proof_eq4}.
\end{proof}

\section{Approximations}\label{sec:approx}

The main result of this section is the following approximation result that is needed to prove Sobolev regularity in Section \ref{sec:reg}.
Its proof is based on several lemmas and will be given in Section~\ref{subsec:proof_thm:approx}.
In Section \ref{sec:comp} we use the approximations to prove comparison principles and apply them to show that solutions corresponing to non-negative initial values remain non-negative.

\begin{proposition}\label{thm:approx}
Let $I \subset \R$ be an open interval with $0 \in I$ and let $\phi:I \to \R$ be a maximal monotone function with $\phi \in W^{1,\infty}_{\loc}(I)$ and $\phi(0)=0$.
Assume, in case $d \geq 3$, that there exists $\alpha \geq \frac{d-2}{d}$ such that
\begin{equation}\label{eq:ex:thm:conv_m-acrr_A_phi;suff}
|\phi(r)| \lesssim_J |r|^{\alpha}, \qquad \forall J \Subset I, r \in J. 
\end{equation}

Let $u_0 \in L^1(\R^d;\overline{I})$, let $f \in L^1([0,T]\times\R^d)$ and let $u \in C([0,T];L^1(\R^d;\overline{I}))$ be the unique mild solution of \eqref{eq:GPME_f_lin}.
Then there exist sequences $(\phi_k)_{k \in \N} \subset C^\infty(\R)$,  $(f_k)_{k \in \N} \subset L^1([0,T] \times \R^d)$, $(u_{0,k})_{k \in \N} \subset L^1(\R^d)$ and $(u_k)_{k \in \N} \subset C^{1,2}([0,T] \times \R^d)$ with the following properties:
\begin{enumerate}[(i)]
    \item\label{it:thm:approx;phi_reg} $\phi_k(0)=0$, $\phi'_k \in BC^\infty(\R)$ and $\inf \phi'_k > 0$ for all $k \in \N$;  
    \item\label{it:thm:approx;class_sol} $u_k$ is a classical solution of
    \begin{equation}\label{eq:approx_eq}
\begin{cases}
\p_t u_k = \Delta \phi_k(u_k)+f_k &\text{in}\:\:(0,T) \times \R^d, \\
u_k(0)=u_{0,k} &\text{on}\:\:\R^d;
\end{cases}
\end{equation}
\item\label{it:thm:approx;mild_sol_L1} $u_k$ belongs to $W^{1,\infty}(0,T;L^1(\R^d))$ and is a strong (and thus mild) solution of \eqref{eq:approx_eq} in the $L^1$-setting;
\item\label{it:thm:approx;decay} $u_k \in C^\infty([0,T];BC^\infty(\R^d))$ with
$$
|\p_t^j\p_x^\alpha u_k(x)| \lesssim_{k,\kappa,|\alpha|} (1+|x|)^{-\kappa}
$$
for all $x \in \R^d$, $j \in \N_0$, $\alpha \in \N^d_0$ and $\kappa \in (d-1,d)$;
\item\label{it:thm:approx;conv_f_u0_u} $f_k \to f$ in $L^1([0,T] \times \R^d)$, $u_{0,k} \to u_0$ in $L^1(\R^d)$ and $u_k \to u$ in $C([0,T];L^1(\R^d))$ as $k \to \infty$;
\item\label{it:thm:approx;phi_conv} $\phi'_k \to \phi'$ in $(L^\infty_\loc (I),\sigma(L^\infty_\loc (I),L^1_c(I)))$ as $k \to \infty$, that is, there is the weak convergence
$$
\int_{\R^d} \phi'(x)h(x)\d x = \lim_{k \to \infty}\int_{\R^d} \phi'_k(x)h(x)\d x, \qquad h \in L^1_c(I).
$$
\end{enumerate}
\end{proposition}

\subsection{Stability}
We now address the approximation of mild solutions using the stability results for $m$-accretive operators in Section \ref{sect_prelim_stab}.
In the particular case that the operators are of the form $A_\phi$ as defined in Proposition \ref{thm:m-accr_A_phi}
one has the following sufficient condition for the convergence of $m$-accretive operators \cite[Theorem~3]{BC81}. 
\begin{proposition}\label{thm:conv_m-acrr_A_phi}
Let $I \subset \R$ be an open interval with $0 \in I$ and let $\phi:I \to \R$ be a maximal monotone function with $\phi(0)=0$. For each $n \in \N$ let $I_n \subset \R$ be an open interval with $0 \in I$ and $\phi_n:I_n \to \R$ be a maximal monotone function with $\phi_n(0)=0$.
Assume that $\phi_n \to \phi$ as $n \to \infty$, that is, $\phi^{\circ} = \lim_{n \to \infty}\phi^{\circ}_n$ a.e.\ on $\R$, where
$$
\phi^{\circ}(r) = \begin{cases}
    -\infty, & r \leq \inf(I),\\ 
    \phi(r), & r \in I,\\
    \infty, & r \geq \sup(I),
\end{cases}
$$
and $\phi^{\circ}_n$ is defined analogously.
Furthermore, if $d \geq 3$, assume that 
\begin{equation}\label{eq:thm:conv_m-acrr_A_phi}
-\int_{R}^{\infty}\rho^{d-1}\beta\left(\frac{-1}{\rho^{d-2}}\right)\d\rho = \int_{R}^{\infty}\rho^{d-1}\beta\left(\frac{1}{\rho^{d-2}}\right)\d\rho = \infty, \qquad R >0,
\end{equation}
where $\beta=\phi^{-1}$ in the sense that $\beta(s) = \min\{r : s=\phi(r)\}$.
Then $A_{\phi_n} \to A_{\phi}$ as m-accretive operators in $L^1(\R^d)$ as $n \to \infty$.
\end{proposition}

Next, we discuss explicit conditions that imply \eqref{eq:thm:conv_m-acrr_A_phi} and show that this property holds for biofilm models.

\begin{proposition}\label{ex:thm:conv_m-acrr_A_phi}
Let $I \subset \R$ be an open interval with $0 \in I$ and let $\phi:I \to \R$ be a maximal monotone function with $\phi(0)=0$.
Let $d \geq 3$ and $\alpha \geq \frac{d-2}{d}$.
If the condition \eqref{eq:ex:thm:conv_m-acrr_A_phi;suff} holds true, then $\phi:I \to \R$ satisfies the condition~\eqref{eq:thm:conv_m-acrr_A_phi}.
\end{proposition}

\begin{proof}
Let $R>0$. We need to show that the two integrals in \eqref{eq:thm:conv_m-acrr_A_phi} diverge. As the proof for both integrals is similar we will only consider the second one.
To this end, put $s_R := \frac{1}{R^{d-2}}$ and $J:=[0,\beta(s_R)]$, where $\beta=\phi^{-1}.$
Then, for each $s \in [0,s_R]$ we have $r=\beta(s) \in J$, such that 
$$
s=\phi(r) \leq C_Jr^{\alpha} = C_J\beta(s)^{\alpha}
$$
and thus,
$$
\beta(s) \geq C_J^{-1/\alpha}s^{1/\alpha}.
$$
Since $d-1-\frac{d-2}{\alpha} \geq -1$, it follows that
$$
\int_{R}^{\infty}\rho^{d-1}\beta\left(\frac{1}{\rho^{d-2}}\right)\d\rho \geq C_J^{-1/\alpha}\int_{R}^{\infty}\rho^{d-1-\frac{d-2}{\alpha}}\d\rho = \infty.
$$
\end{proof}

\begin{example}\label{ex:thm:conv_m-acrr_A_phi;biofilm}
The maximal monotone function $\phi:(-1,1) \to \R$ from Example~\ref{ex:max_mon_fc_biofilm} satisfies the condition~\eqref{eq:ex:thm:conv_m-acrr_A_phi;suff} with $\alpha=b+1$.

Indeed, to see that $\phi$ satisfies the condition~\eqref{eq:ex:thm:conv_m-acrr_A_phi;suff} with $\alpha=b+1$, let $R \in (0,1)$ and note that, for each $\rho \in [-R,R]$,
$$
|\phi(\rho)| \leq \int_{0}^{|\rho|}\frac{z^b}{(1-z)^a}\d z \leq C_R\int_{0}^{|\rho|}z^b \d z 
= \frac{1}{b+1}|\rho|^{b+1},
$$
for some constant $C_R>0$ depending on $R$.
\end{example}

Combining Propositions \ref{thm:conv_m-acrr_A_phi}, \ref{prop:A_phi_closure_dom} and Theorem \ref{thm:nonlin_Trotter-Kato} immediately yields the following stability result.

\begin{corollary}\label{cor:thm:nonlin_Trotter-Kato_A_phi}
Let the assumptions in Theorem~\ref{thm:conv_m-acrr_A_phi} hold.  
For each $n \in \N$, let $f^n \in L^1(0,T;L^1(\R^d))$, $y^n_0 \in L^1(\R^d)$ with $y^n_0 \in \overline{I_n}$ a.e.\ and let $y_n$ be the mild solution to 
\begin{equation}\label{eq:abstract_Cauchy_n_A_phi}
    \begin{cases}
        y_n'(t)+A_{\phi_n}y_n(t) = f^n(t), & t \in [0,T],\\
        y_n(0) = y^n_0.
    \end{cases}
\end{equation}
If $f=\lim_{n \to \infty} f^n$ in $L^1(0,T;L^1(\R^d))$ and $y_0=\lim_{n \to \infty}y^n_0$ in $L^1(\R^d)$, then $y_n \to y$ in $C([0,T];L^1(\R^d))$ as $n \to \infty$, where $y$ is the mild solution to
\begin{equation}\label{eq:abstract_Cauchy_A_phi}
    \begin{cases}
        y'(t)+A_\phi y(t) = f(t), & t \in [0,T],\\
        y(0) = y_0.
    \end{cases}
\end{equation}
\end{corollary}

\begin{remark}
Note that the mild solutions $y_n$ in Theorem~\ref{thm:nonlin_Trotter-Kato} and Corollary~\ref{cor:thm:nonlin_Trotter-Kato_A_phi} exist
by Theorem~\ref{thm:exist_mild_sol_Cauchy}.    
\end{remark}

\subsection{Mild solutions of non-degenerate problems}

In this subsection we prove that the notions of mild and classical solutions coincide for smooth, non-degenerate quasilinear problems, 
cf.\ Remark~\ref{rmk:mild_solution_concept_strong}.

\begin{proposition}\label{lem:classical_implies_mild}
Let $\psi:\R \to \R$ be a $C^2$-function with $\psi(0)=0$, $\psi', \psi'' \in L^\infty(\R)$ and $\psi'(r)>0$ for all $r \in \R$.   Let $v \in C^{1,2}([0,T] \times \R^d)$ be a classical solution of 
\begin{equation}\label{eq:lem:L1-contraction;PDE}
\p_t v = \Delta \psi(v)+g\quad\text{in}\:\:(0,T) \times \R^d,
\end{equation}
with $\p_t\nabla_x v \in C([0,T] \times \R^d)$ and $\p_t v \in L^\infty([0,T] \times \R^d)$ for a given $g \in W^{1,1}((0,T) \times \R^d)$. 
Furthermore, assume that 
\begin{align}
\lim_{R \to \infty}\int_{\p B_R(0)}|\nabla v(t)|\d\sigma 
&=\lim_{R \to \infty}\int_{\p B_R(0)}|\p_t\nabla v(t)|\d\sigma = 0, \label{eq:lem:classical_implies_mild;decay}
\end{align}
and $\Delta\psi(v(0))+g(0) \in L^1(\R^d)$.
If $v$ satisfies the initial condition $v(0)=v_0$, then $v$ belongs to $W^{1,\infty}(0,T;L^1(\R^d))$ and is a strong (and thus mild) solution of \eqref{eq:GPME;psi_v} in the $L^1(\R^d)$-setting.    
\end{proposition}

The following lemma provides an $L^1$-contraction principle for \eqref{eq:lem:L1-contraction;PDE}. 
Whereas the final statement  \eqref{eq:lem:L1-contraction} is closest to a classical formulation of an  $L^1$-contraction principle, the estimates \eqref{eq:lem:L1-contraction;ball_intermediate} and \eqref{eq:lem:L1-contraction;ball} turn out to be important for the proof of Proposition~\ref{lem:classical_implies_mild}. 

\begin{lemma}[$L^1$-contraction principle]\label{lem:L1-contraction}
Let $\psi:\R \to \R$ be a $C^2$-function with $\psi(0)=0$, $\psi' \in L^\infty(\R)$ and $\psi'(r)>0$ for all $r \in \R$.   Let $v, \tilde{v} \in C^{1,2}([0,T] \times \R^d)$ be classical solutions of
\begin{equation}
\p_t v = \Delta \psi(v)+g\quad\text{in}\:\:(0,T) \times \R^d,
\end{equation}
and 
\begin{equation*}
\p_t \tilde{v} = \Delta \psi(\tilde{v})+\tilde{g}\quad\text{in}\:\:(0,T) \times \R^d, 
\end{equation*}
respectively, where $g,\tilde{g} \in L^1([0,T] \times \R^d)$.
Then, for every $t \in [0,T]$ and $R>0$, we have the estimates
\begin{align}
&\|(v(t,\cdot)-\tilde{v}(t,\cdot))_+\|_{L^1(B_R(0))}\nonumber \\
\leq& \|(v(0,\cdot)-\tilde{v}(0,\cdot))_+\|_{L^1(B_R(0))} + \|(g-\tilde{g})_+ \sign^+_0(v-\tilde{v})\|_{L^1((0,t)\times B_R(0))} \label{eq:lem:L1-contraction;ball_intermediate;+}   \\
&+ \|\nabla(\psi(v)-\psi(\tilde{v}))\|_{L^1(\partial B_R(0))} 
\nonumber \\
\leq& \|(v(0,\cdot)-\tilde{v}(0,\cdot))_+\|_{L^1(B_R(0))} + \|(g-\tilde{g})_+ \sign^+_0(v-\tilde{v})\|_{L^1((0,t)\times B_R(0))} \label{eq:lem:L1-contraction;ball;+}    \\
&+ \nrm{\psi'}_{L^\infty(\R)}\left(\|\nabla v(t,\cdot)\|_{L^1(\partial B_R(0))}  + \|\nabla \tilde{v}(t,\cdot)\|_{L^1(\partial B_R(0))}  \right).
\nonumber
\end{align}
and, as consequence, also the estimates
\begin{align}
&\|v(t,\cdot)-\tilde{v}(t,\cdot)\|_{L^1(B_R(0))} \nonumber \\
\leq& \|v(0,\cdot)-\tilde{v}(0,\cdot)\|_{L^1(B_R(0))} + \|g-\tilde{g}\|_{L^1((0,t)\times B_R(0))} \label{eq:lem:L1-contraction;ball_intermediate}  \\
& +  \|\nabla(\psi(v)-\psi(\tilde{v}))\|_{L^1(\partial B_R(0))} 
\nonumber \\
\leq&  \|v(0,\cdot)-\tilde{v}(0,\cdot)\|_{L^1(B_R(0))} + \|g-\tilde{g}\|_{L^1((0,t)\times B_R(0))}\label{eq:lem:L1-contraction;ball} 
   \\
&+ \nrm{\psi'}_{L^\infty(\R)}\left(\|\nabla v(t,\cdot)\|_{L^1(\partial B_R(0))}  + \|\nabla \tilde{v}(t,\cdot)\|_{L^1(\partial B_R(0))}  \right).
\nonumber
\end{align}
In particular, if $\int_{\p B_R(0)}|\nabla v(t)|\d\sigma \to 0$ and $\int_{\p B_R(0)}|\nabla \tilde{v}(t)|\d\sigma \to 0$ for $R \to \infty$, then
\begin{align}\label{eq:lem:L1-contraction;+}
\begin{split}
&\nrm{(v(t)-\tilde{v}(t))_+}_{L^1(\R^d)} \\
\leq& \nrm{(v(0)-\tilde{v}(0))_+}_{L^1(\R^d)} + \nrm{(g-\tilde{g})_+\sign^+_0(v-\tilde{v})}_{L^1([0,t] \times \R^d)}.
\end{split}
\end{align} 
and, as a consequence, also
\begin{equation}\label{eq:lem:L1-contraction}
\nrm{v(t)-\tilde{v}(t)}_{L^1(\R^d)} \leq \nrm{v(0)-\tilde{v}(0)}_{L^1(\R^d)} + \nrm{g-\tilde{g}}_{L^1([0,t] \times \R^d)}.
\end{equation}
\end{lemma}

\begin{proof}
This follows by a modification of the proof of
\cite[Proposition~3.5]{Va07} for bounded domains $\Omega$ with homogeneous Dirichlet boundary conditions.
Indeed, let $0 \leq \eta\leq 1$ be a $C^1$-approximation of the $\sign^+_0$ function and $w=\psi(v)-\psi(\tilde{v})$.  
Then, integration by parts yields
$$
\int_\Omega \Delta w \,\eta(w)\,\d x = -\int_\Omega |\nabla w|^2 \eta'(w)\,\d x + \int_{\p \Omega}\nabla w \cdot \nu \, \eta(w)\,\d\sigma,
$$
where $\nu$ denotes the outward unit normal vector on $\p \Omega$.  
The boundary term can be estimated by
\begin{align*}
|\nabla w \cdot \nu\, \eta(w)| &\leq |\nabla w| = |\nabla(\psi(v)-\psi(\tilde{v}))| = 
|\psi'(v)\nabla v-\psi'(\tilde{v})\nabla\tilde{v}| \\
&\leq |\psi'(v)\nabla v| + |\psi'(\tilde{v})\nabla\tilde{v}|
\leq \nrm{\psi'}_{L^\infty(\R)}(|\nabla v| + |\nabla\tilde{v}|).
\end{align*}
With this modification the argument in the proof of \cite[Proposition~3.5]{Va07} yields \eqref{eq:lem:L1-contraction;ball_intermediate;+} and \eqref{eq:lem:L1-contraction;ball;+} by taking $\Omega=B_R(0)$. The estimate \eqref{eq:lem:L1-contraction;+} subsequently follows by monotone convergence.
\end{proof}

Next we use the above $L^1$-contraction principle to show that $v(t)$ and $\partial_t v(t)$ belong to $L^1(\R^d)$ with concrete norm estimates. 

\begin{lemma}\label{lem:L1_est_v_time_der}
Let the assumptions of Proposition~\ref{lem:classical_implies_mild} be satisfied.
Then $v(t),\partial_t v(t) \in L^1(\R^d)$ with norm estimates
\begin{equation}\label{lem:L1_est_v_time_der;v}
\nrm{v(t)}_{L^1(\R^d)} \leq \nrm{v(0)}_{L^1(\R^d)} + \nrm{g}_{L^1([0,T] \times \R^d)},  
\end{equation}
and
\begin{equation}\label{lem:L1_est_v_time_der;dtv}
\nrm{\partial_t v(t)}_{L^1(\R^d)} \leq \nrm{\Delta\psi(v(0))+g(0)}_{L^1(\R^d)} + \nrm{\partial_t g}_{L^1([0,T] \times \R^d)}.  
\end{equation} 
\end{lemma}

\begin{proof}
The estimate \eqref{lem:L1_est_v_time_der;v} follows directly from \eqref{eq:lem:L1-contraction;ball} by taking $\tilde{v}=0$ 
and $\tilde g=0$
in Lemma~\ref{lem:L1-contraction}. 
Applying Lemma~\ref{lem:L1-contraction} with $\tilde{v}=v(\,\cdot\,+h,\,\cdot\,)$ gives through \eqref{eq:lem:L1-contraction;ball_intermediate}
\begin{align*}
&\int_{B_R(0)}\left|\frac{v(t+h,x)-v(t,x)}{h}\right|\d x \\
\leq& \int_{B_R(0)}\left|\frac{v(h,x)-v(0,x)}{h}\right|\d x + \int_{0}^T\int_{B_R(0)}\left|\frac{g(s+h,x)-g(s,x)}{h}\right|\d x \d s  \\
& + \int_{\p B_R(0)}\left|\frac{\nabla(\psi(v(t+h)))-\nabla(\psi(v(t)))}{h}\right|\d\sigma .
\end{align*}
Taking the limit $h \to 0$ we obtain by dominated convergence
\begin{align*}
\int_{B_R(0)}|\p_t v(t)|\d x 
\leq& \int_{B_R(0)}|\p_t v(0)|\d x 
+ \int_{0}^T\int_{B_R(0)}|\p_t g(s,x)|\d x \d s  \\
&+ \int_{\p B_R(0)}|\p_t\nabla (\psi(v(t)))|\d\sigma .
\end{align*}
Observing that
\begin{align*}
|\p_t\nabla (\psi(v))| &=
|\p_t(\psi'(v)\nabla v)| =
|\psi''(v)\p_t v \nabla v +\psi'(v)\p_t\nabla v| \\
&\leq \nrm{\psi''}_{L^\infty(\R)}
\|\p_t v\|_{L^\infty([0,T] \times \R^d)}|\nabla v| +\nrm{\psi'}_{L^\infty(\R)}|\p_t\nabla v|,
\end{align*}
and taking the limit $R \to \infty$, we obtain by monotone convergence that
$$
\int_{\R^d}|\p_t v(t)|\d x 
\leq \int_{\R^d}|\p_t v(0)|\d x 
+ \int_{0}^T\int_{\R^d}|\p_t g(s,x)|\d x \d s.
$$
Using that $v$ is a classical solution of \eqref{eq:lem:L1-contraction;PDE} we finally arrive at \eqref{lem:L1_est_v_time_der;dtv}.
\end{proof}

Note that  Lemma~\ref{lem:L1_est_v_time_der} implies that $v(t),\partial_t v(t) \in L^1(\R^d)$. This together with the given norm estimates is already quite close to $v$ being a strong solution in the $L^1(\R^d)$-setting, i.e.\ the conclusion of Proposition~\ref{lem:classical_implies_mild}.
However, regularity in the form of strong measurablity with respect to $t$ as an $L^1(\R^d)$-valued function is missing.
This technicality is taken care of in the following proof.

\begin{proof}[Proof of Proposition~\ref{lem:classical_implies_mild}]
By Lemma~\ref{lem:L1_est_v_time_der} we can view $v$ as a function $v:[0,T] \to L^1(\R^d)$.
Now note that, for all $h \in L^\infty_c(\R^d)$, 
\begin{equation*}\label{eq:lem:classical_implies_mild;pairing}
t \mapsto \ip{v(t),h} = \int_{\R^d}v(t,x)h(x) \d x    
\end{equation*}
is a continuous and thus measurable function.
As $L^1(\R^d)$ is a separable Banach space and $L^\infty_c(\R^d)$ is a weak$^*$ dense subspace of $[L^1(\R^d)]^*=L^\infty(\R^d)$,
Pettis' measurability theorem (see \cite[Theorem~1.1.6]{HNVW16}) yields that $v:[0,T] \to L^1(\R^d)$ is strongly measurable.
In the same way, it follows that $\partial_t v:[0,T] \to L^1(\R^d)$ is a strongly measurable function.
It is not difficult to see that $\partial_t v:[0,T] \to L^1(\R^d)$ is the weak derivative of $v:[0,T] \to L^1(\R^d)$, and subsequently, that $v$ belongs to $W^{1,\infty}(0,T;L^1(\R^d))$ and is a strong solution of
\eqref{eq:GPME;psi_v}.
\end{proof}

\subsection{Decay estimates for  non-degenerate problems}

In this subsection we prove decay estimates. The  main aim is to provide sufficient conditions that imply the assumptions of Proposition~\ref{lem:classical_implies_mild}, particularly, the decay assumption \eqref{eq:lem:classical_implies_mild;decay}.
Our strategy is to follow an approach based on weighted $L^p$-spaces.

For $p \in [1,\infty]$ and $\kappa \in \R$ we consider the weight 
\begin{equation}
w_\kappa(x):=\max\{1,|x|^\kappa\} = \begin{cases}
1,& |x| < 1,\\ 
|x|^\kappa, & |x| \geq 1,
\end{cases}    
\end{equation}
on $\R^d$. For an open set $U \subset \R^d$ the
associated weighted $L^p$-space $L^p_\kappa(U)$ is defined by
\begin{equation}
L^p_\kappa(U) := \left\{ f \in L^0(U) : w_\kappa f \in L^p(U) \right\}, \quad \norm{f}_{L^p_\kappa(U)} := \norm{w_\kappa f}_{L^p(U)},   
\end{equation}
and, for $k \in \N$, the corresponding weighted Sobolev space $W^{k,p}_\kappa(U)$ by
\begin{align*}
W^{k,p}_\kappa(U) &:= \{ f \in \ms{D}'(U) : \partial^\alpha f \in L^p_\kappa(U), |\alpha| \leq k \}, \\
\norm{f}_{W^{k,p}_\kappa(U)} &:= \sum_{|\alpha| \leq k}\norm{\partial^\alpha f}_{L^p_\kappa(U)}.    
\end{align*}
Furthermore, we define
\begin{align*}
BC^{k}_\kappa(U) &:=  W^{k,\infty}_\kappa(U) \cap C^k(U), \\
\norm{f}_{BC^{k}_\kappa(U)} &:= \norm{f}_{W^{k,\infty}_\kappa(U)}.    
\end{align*}

\begin{proposition}\label{prop:weighted_Sob_emb}
Let $k \in \N$, $p \in (d,\infty)$ and $\kappa \in \R$. Then we have the continuous embedding
\begin{equation*}
W^{k,p}_\kappa(\R^d) \hookrightarrow BC^{k-1}_\kappa(\R^d).    
\end{equation*}
\end{proposition}

\begin{proof}
By induction, it suffices to consider the case $k=1$.
Furthermore, by density of $\ms{S}(\R^d)$ in $W^{1,p}_\kappa(\R^d)$ it suffices to show that
\begin{equation*}
W^{1,p}_\kappa(\R^d) \hookrightarrow L^\infty_\kappa(\R^d).    
\end{equation*}
By the classical Sobolev embedding theorem (see e.g.\ \cite[Corollary~9.14]{Br11}) and a translation argument we have
$$
W^{1,p}(B_1(x_0)) \hookrightarrow L^\infty(B_1(x_0)), \qquad x_0 \in \R^d,
$$
with a norm estimate independent of $x_0$.
Since $w_\kappa(x) \eqsim w_\kappa(x_0)$ 
for all $x_0 \in \R^d$ and $x \in B_1(x_0)$, it follows that
\begin{align*}
\norm{f}_{L^\infty_\kappa(B_1(x_0))}
&\eqsim w_\kappa(x_0)\norm{f}_{L^\infty(B_1(x_0))}
\lesssim w_\kappa(x_0)\norm{f}_{W^{1,p}(B_1(x_0))} \\
&\eqsim \norm{f}_{W^{1,p}_\kappa(B_1(x_0))}
\leq \norm{f}_{W^{1,p}_\kappa(\R^d)}.
\end{align*}
By the nature of the $L^\infty$-norm and since $x_0$ is arbitrary this implies that $\norm{f}_{L^\infty_\kappa(\R^d)} \lesssim \norm{f}_{W^{1,p}_\kappa(\R^d)}$.
\end{proof}

\begin{theorem}\label{thm:decay_at_infty_sol}
Let $\psi \in C^3(\R)$ satisfy $\psi',\psi'', \psi''' \in L^\infty(\R)$ and $\psi' \geq c$ for some $c \in (0,\infty)$. Let $p \in (d,\infty)$, $\kappa \in \big[0,d(1-\frac{1}{p})\big)$, $\alpha \in (0,\frac{1}{2}]$ and $\beta \in (0,1)$.
Let $g \in C^\alpha([0,T];L^{p}_\kappa(\R^d)\cap BC^{\beta}(\R^d))$, $v_0 \in W^{2,p}_\kappa(\R^d) \cap BC^{\beta+2}(\R^d)$ and  $v \in C^1([0,T];BC^\beta(\R^d)) \cap C([0,T];BC^{\beta+2}(\R^d))$ be a solution of the problem
\begin{align}\label{eq:thm:decay_at_infty_sol:Eq}
\begin{cases}
v_t = \Delta \psi(v)+g &\text{in}\ (0,T]\times \R^d,\\
v(0)=v_0&\text{on}\   \R^d.    
\end{cases}
\end{align}
Then
\begin{equation}\label{eq:thm:decay_at_infty_sol;C1C}
v \in C^1([0,T];L^{p}_\kappa(\R^d)) \cap C([0,T];W^{2,p}_\kappa(\R^d)).    
\end{equation}
\end{theorem}

\begin{proof}
Note that the equation \eqref{eq:thm:decay_at_infty_sol:Eq} can be rewritten as
$$
v_t = \psi'(v)\Delta v + \psi''(v)\nabla v \cdot \nabla v + g.
$$
In particular, we can view $v$ as a solution $u=v$ of 
\begin{equation}\label{eq:thm:decay_at_infty_sol;div}
u_t = \psi'(v)\Delta u + \psi''(v)\nabla v \cdot \nabla u + g,
\end{equation}
which is a non-autonomous linear equation in $u$.

By \cite[Proposition~1.1.4, Proposition~1.2.13 and Corollary~1.2.18]{Lu95} there is the mixed-derivative embedding
\begin{equation*}
C^1([0,T];BC^\beta(\R^d)) \cap C([0,T];BC^{\beta+2}(\R^d)) \hra C^{\frac{1}{2}}([0,T];BC^{\beta+1}(\R^d)),
\end{equation*}
so in particular,
\begin{equation}\label{eq:thm:decay_at_infty_sol;mixed-reg_v}
v \in C^{\frac{1}{2}}([0,T];BC^{\beta+1}(\R^d)) \hra C^{\alpha}([0,T];BC^{\beta+1}(\R^d)).           
\end{equation}

We define
$$
X := BC^{\beta}(\R^d), \qquad D := BC^{\beta+2}(\R^d),
$$
and, for each $t \in [0,T]$, consider the linear operator
$$
A(t):D \to X,\quad \: u \mapsto \psi'(v(t))\Delta u + \psi''(v(t))\nabla v(t) \cdot \nabla u.
$$
Note here that, thanks to the embedding \eqref{eq:thm:decay_at_infty_sol;mixed-reg_v} and the assumption $\psi\in C^3(\R)$, $\psi',\psi'',\psi''' \in L^\infty(\R^d)$, we have $\psi'(v), \psi''(v)\nabla v \in C^\alpha([0,T];BC^\beta(\R^d))$.
By \cite[Theorem~3.1.14 and Corollary~3.1.16]{Lu95} the operator-theoretic conditions of \cite[Section~6.1,pg.~212]{Lu95} are thus satisfied.
As $g \in C^\alpha([0,T];X)$ and $v_0 \in D$, \cite[Corollary~6.1.6 and Proposition~6.2.2]{Lu95} yields that there exists a unique strict solution $u$ of
\begin{equation}\label{eq:thm:decay_at_infty_sol;abstract_eq}
u'(t) = A(t)u(t)+g(t),\:\: 0 < t \leq T; \quad u(0)=v_0.    
\end{equation}
A strict solution satisfies $u \in C^1([0,T];X) \cap C([0,T];D)$ and solves the initial value problem \eqref{eq:thm:decay_at_infty_sol;abstract_eq}. 
As $v$ is a strict solution of this equation as well, we conclude that $v=u$.

Next we define
$$
\tilde{X} := BC^{\beta}(\R^d) \cap L^p_\kappa(\R^d), \qquad \tilde{D} := BC^{\beta+2}(\R^d) \cap W^{2,p}_\kappa(\R^d),
$$
and, for each $t \in [0,T]$, consider the linear operator
$$
\tilde{A}(t):\tilde{D} \to \tilde{X},\quad u \mapsto \psi'(v(t))\Delta u + \psi''(v(t))\nabla v(t) \cdot \nabla u,
$$
where we have, as above, $\psi'(v), \psi''(v)\nabla v \in C^\alpha([0,T];BC^\beta(\R^d))$.
By \cite[Theorem~3.1.4 and Corollary~3.1.6]{Lu95} and \cite[Theorem~3.1]{HHH03} the operator-theoretic conditions of \cite[Section~6.1,pg.~212]{Lu95} are thus satisfied.
As $g \in C^\alpha([0,T];\tilde{X})$ and $v_0 \in \tilde{D}$, \cite[Corollary~6.1.6 and Proposition~6.2.2]{Lu95} implies that there exists a unique strict solution $\tilde{u} \in C^1([0,T];\tilde{X}) \cap C([0,T];\tilde{D})$ of 
\begin{equation}\label{eq:thm:decay_at_infty_sol;abstract_eq_tilde}
\tilde{u}'(t) = \tilde{A}(t)\tilde{u}(t)+g(t),\quad 0 < t \leq T; \qquad \tilde{u}(0)=v_0. 
\end{equation}
As $\tilde{X} \hra X$, $\tilde{D} \hra D$ and $A(t)|_{\tilde{D}}=\tilde{A}(t)$ for every $t \in [0,T]$, every strict solution of \eqref{eq:thm:decay_at_infty_sol;abstract_eq_tilde} is a strict solution of \eqref{eq:thm:decay_at_infty_sol;abstract_eq}.
Therefore, by uniqueness of strict solutions,  $v=u=\tilde{u} \in C^1([0,T];\tilde{X}) \cap C([0,T];\tilde{D})$.
\end{proof}

Next, we generalize the result to obtain higher regularity. 

\begin{theorem}\label{thm:decay_at_infty_sol_W1p}
Let $k \in \N$, $K:= \max\{k+2,3\}$ and $\psi \in C^{K}(\R)$ satisfy $\psi',\psi'', \ldots, \psi^{[K]} \in L^\infty(\R)$ and $\psi' \geq c$ for some $c \in (0,\infty)$.
Let $p \in (d,\infty)$, $\kappa \in [0,d(1-\frac{1}{p}))$, $\alpha \in (0,\frac{1}{2}]$ and $\beta \in (0,1)$.
Let $g \in C^\alpha([0,T];W^{k,p}_\kappa(\R^d)\cap BC^{\beta}(\R^d))$, $v_0 \in W^{k+2,p}_\kappa(\R^d) \cap BC^{\beta+2}(\R^d)$ and $v \in C^1([0,T];BC^\beta(\R^d)) \cap C([0,T];BC^{\beta+2}(\R^d))$ be a solution of \eqref{eq:thm:decay_at_infty_sol:Eq}.
Then
\begin{equation}\label{eq:thm:decay_at_infty_sol_W1p}
v \in C^1([0,T];W^{k,p}_\kappa(\R^d)) \cap C([0,T];W^{k+2,p}_\kappa(\R^d)).    
\end{equation}
\end{theorem}

For the proof of this theorem we will need the following lemma.

\begin{lemma}\label{lem:W1p_operator}
Let $p \in (d,\infty)$, $\kappa \in (-\frac{d}{p},d(1-\frac{1}{p}))$ and $k \in \N$. 
Let $\rho \in W^{k,\infty}(\R^d) \cap C(\R^d)$ with $\lim_{|x| \to \infty}\rho(x) = \rho(\infty)$ satisfy $\rho \geq c$ for some $c>0$ and let $\mu \in W^{k,p}(\R^d)$.
Then the operator $A$ on $W^{k,p}_\kappa(\R^d)$ given by
\begin{equation*}
\dom(A) = W^{k+2,p}_\kappa(\R^d), \qquad Au = \rho\Delta u + \mu \,\cdot\,\nabla u,    
\end{equation*}
is sectorial in the sense of \cite[Definition~2.0.1]{Lu95}.
\end{lemma}

In the proof of this lemma we will treat $b \,\cdot\,\nabla u$ as a lower order perturbation, for which we need a version of Proposition~\ref{prop:weighted_Sob_emb} for Sobolev spaces with fractional smoothness.

In order to introduce the setting, let us first note that in the notation of \cite{MV12} we have
$$
W^{k,p}_\kappa(\R^d) = W^{k,p}(\R^d,w_{\kappa}^p) = W^{k,p}(\R^d,w_{\kappa p}).
$$
This leads us to defining the corresponding fractional Sobolev spaces of Bessel potential type $H^{s,p}_{\kappa}$ as follows. Given $p \in (1,\infty)$, $s \in \R$ and $\kappa \in (-\frac{d}{p},d(1-\frac{1}{p}))$, we define 
$$
H^{s,p}_{\kappa}(\R^d) := H^{s,p}(\R^d,w_{\kappa p}),
$$
where $H^{s,p}(\R^d,w_{\kappa p})$ is as in \cite[Definition~3.7]{MV12}.
Here the condition $\kappa \in (-\frac{d}{p},d(1-\frac{1}{p}))$, or equivalently, $\kappa p \in (-d,d(p-1))$ coincides with the so-called Muckenhoupt $A_p$-condition for the weight $w_{\kappa p}$, see \cite[Section~4.1]{MV14}/\cite[Proposition~2.6]{HS08}.

Below we will use weighted Besov spaces.
These are defined in the same way as the classical unweighted Besov spaces introduced in Section~\ref{subsec:function_spaces}, simply by replacing $L^p$ by its corresponding weighted version. For details we refer the reader to \cite[Section~3.1]{MV12}.

\begin{proposition}\label{prop:weighted_Sob_emb_frac}
Let $p \in (1,\infty)$, $\kappa \in (-\frac{d}{p},d(1-\frac{1}{p}))$, $k \in \N$ and $s \in (0,1)$.
If $s>\frac{d}{p}$, then
\begin{equation}\label{eq:prop:weighted_Sob_emb_frac}
H^{k+s,p}_\kappa(\R^d) \hra BC^k_\kappa(\R^d).    
\end{equation}
\end{proposition}

As we will discuss in the proof of Proposition~\ref{prop:weighted_Sob_emb_frac}, the embedding \eqref{eq:prop:weighted_Sob_emb_frac} follows from the embedding
$$
H^{s,p}_\kappa(\R^d) \hra L^{\infty}_\kappa(\R^d).
$$
Moreover, by the elementary embeddings \cite[Propositions 3.11 and 3.12]{MV12}, it suffices to prove that
$$
B^{\sigma}_{p,1}(\R^d,w_{\kappa p}) \hra L^{\infty}_\kappa(\R^d)
$$
for $\sigma \in (\frac{d}{p},s)$.
Embeddings of this type are studied in \cite[Proposition~7.1]{MV12}.
However, our result is not included there due to a difference in our definition of weighted $L^r$-spaces.
Indeed, \cite{MV12} uses a change of measure approach whereas we use a multiplier approach.
For $r<\infty$ the two are equivalent, but for
   $r=\infty$ the former does not lead to anything more than the usual unweighted $L^\infty$-spaces and in that sense does not include the case $r=\infty$ (see \cite[Remark~3.20]{LN23}).

\begin{proof}[Proof of Proposition~\ref{prop:weighted_Sob_emb_frac}]
By density of the Schwartz space $\ms{S}(\R^d)$ in $H^{k+s,p}_\kappa(\R^d)$ it suffices to show that
$$
H^{k+s,p}_\kappa(\R^d) \hra W^{k,\infty}_\kappa(\R^d).
$$
Moreover, it suffices to consider the case $k=0$.
Pick $\theta \in (\frac{d}{p},s)$ and put $\varepsilon := s-\theta > 0$.
Define $p_1 := \theta p$ and $\kappa_1 := \frac{\kappa}{\theta}$. Then $p_1 > \frac{d}{p}p=d$ and
$$
\kappa_1=\frac{p}{p_1}\kappa > -\frac{p}{p_1}\frac{d}{p}=-\frac{d}{p_1}.
$$
In particular, $w_{\kappa_1 p_1}$ satisfies the Muckenhoupt $A_{\infty}$-condition.

By \cite[Proposition~2.5.7]{Tr83} we have 
$$
B^0_{\infty,1}(\R^d) \hra L^\infty(\R^d). 
$$
Combining Proposition~\ref{prop:weighted_Sob_emb} with the elementary embedding $B^1_{p_1,1}(\R^d,w_{\kappa_1 p_1}) \hra W^{1,p_1}(\R^d,w_{\kappa_1 p_1})$ from \cite[Proposition~3.12]{MV12} we obtain that 
$$
B^1_{p_1,1}(\R^d,w_{\kappa_1 p_1}) \hra L^\infty_{\kappa_1}(\R^d).
$$
Complex interpolation thus yields
$$
[B^0_{\infty,1}(\R^d),B^1_{p_1,1}(\R^d,w_{\kappa_1 p_1})]_\theta \hra [L^\infty(\R^d),L^\infty_{\kappa_1}(\R^d)]_\theta.
$$
As in the proof of \cite[Theorem~5.5.3]{BL76} it can be shown that
$$
[L^\infty(\R^d),L^\infty_{\kappa_1}(\R^d)]_\theta = L^\infty_\kappa(\R^d).
$$
Furthermore, by \cite[Theorem~4.5]{SSV14}, we have
$$
B^{\theta}_{p,1}(\R^d,w_{\kappa p}) = [B^0_{\infty,1}(\R^d),B^1_{p_1,1}(\R^d,w_{\kappa_1 p_1})]_\theta.
$$
The above embedding can thus be rewritten as
$$
B^{\theta}_{p,1}(\R^d,w_{\kappa p}) \hra L^\infty_\kappa(\R^d).
$$
Since 
$$
H^{s,p}_\kappa(\R^d) = H^{\theta+\varepsilon,p}(\R^d,w_{\kappa p}) \hra B^{\theta+\varepsilon}_{p,\infty}(\R^d,w_{\kappa p}) \hra B^{\theta}_{p,1}(\R^d,w_{\kappa p})
$$
by \cite[Propositions 3.11 and 3.12]{MV12}, the desired embedding follows.
\end{proof}

\begin{proof}[Proof of Lemma~\ref{lem:W1p_operator}]
The case that $\rho$ is constant and $\mu=0$ follows from \cite[Theorem~5.1]{GV17} by a standard lifting argument to reduce to the case $k=0$.  
A standard localization argument subsequently yields the case of a non-constant $\rho$ and $\mu=0$, see for instance the discussion in \cite[Appendix~C]{HL22}.
Finally, the case of non-trivial $\mu$ can be obtained by lower order perturbation (using for instance \cite[Proposition~2.4.1(i)]{Lu95}).
Indeed, picking $\theta \in (0,1)$ with $\theta > \frac{1}{2}(\frac{d}{p}+1)$, we have
$$
[W^{k,p}_\kappa(\R^d),W^{k+2,p}_\kappa(\R^d)]_\theta = H^{k+2\theta,p}_\kappa(\R^d) \stackrel{\nabla}{\longrightarrow}H^{k+2\theta-1,p}_\kappa(\R^d) \hra W^{k,\infty}_\kappa(\R^d)
$$
by Proposition~\ref{prop:weighted_Sob_emb_frac}, while pointwise multiplication maps as
$$
W^{k,p}(\R^d) \times W^{k,\infty}_\kappa(\R^d) \longrightarrow W^{k,p}_\kappa(\R^d).
$$
\end{proof}

In the following lemma we study composition operators on Sobolev spaces (cf.\ \cite[Section~5.2.4, Theorem~1]{RS96})).
\begin{lemma}\label{lem:comp_operator}
Let $k \in \N_0$ and $p \in (d,\infty)$.  
For every $F \in BC^{k+1}(\R^d)$ and $v \in W^{k+1,p}(\R^d)$, we have $\nabla (F \circ v) \in W^{k,p}(\R^d)$. 
\end{lemma}
Let us remark that for the stronger conclusion $F \circ v \in W^{k+1,p}(\R^d)$ it is necessary that $F(0)=0$ (see \cite[Section~5.2.4]{RS96}).
Taking the gradient allows us to omit this assumption.

\begin{proof}
By the classical Sobolev embedding theorem (see e.g.\ \cite[Corollary~9.14]{Br11}) we have
\begin{equation}\label{eq:lem:comp_operator;Sob_emb}
W^{1,p}(\R^d) \hra BC(\R^d) \hra L^\infty(\R^d).    
\end{equation}
Using this embedding and performing a density argument to reduce to the case $v \in \ms{S}(\R^d)$, the statement follows from the formula of Fa\`a di Bruno.
\end{proof}

\begin{proof}[Proof of Theorem \ref{thm:decay_at_infty_sol_W1p}]
We proceed by induction on $k$. Below we will use the notation
$$
X_l := W^{l,p}_\kappa(\R^d), \qquad 
D_{l} := X_{l+2} = W^{l+2,p}_\kappa(\R^d), \qquad l \in \N.
$$
Furthermore, we will use that $\{X_l\}_{l \in \N}$ is a complex interpolation scale (which follows e.g. from\ \cite[Propositions~3.2 and 3.7]{MV15}).

The case $k=0$ is covered by Theorem~\ref{thm:decay_at_infty_sol}.
Now assume the theorem holds true for some $k \in \N$ and suppose the assumptions of the theorem are satisfied with $k+1$ in place of $k$.

By \cite[Proposition~1.1.4 and Proposition~1.2.13]{Lu95} there is the mixed-derivative embedding
\begin{equation*}
C^1([0,T];X_k) \cap C([0,T];D_k) \hra C^{\frac{1}{2}}([0,T];W^{k+1,p}_\kappa(\R^d)).
\end{equation*}
As $v$ satisfies \eqref{eq:thm:decay_at_infty_sol_W1p} and $\alpha \leq \frac{1}{2}$, we thus have
\begin{equation*}
v \in C^{\alpha}([0,T];W^{k+1,p}_\kappa(\R^d))     \hra C^{\alpha}([0,T];W^{k+1,p}(\R^d)).
\end{equation*}
By the Sobolev embedding \eqref{eq:lem:comp_operator;Sob_emb} and the assumptions on $\psi$ this implies that $\psi'(v) \in  C^{\alpha}([0,T];BC^{k}(\R^d))$.
Furthermore, by Lemma~\ref{lem:comp_operator}, $\psi''(v)\nabla v = \nabla(\psi'(v)) \in  C^{\alpha}([0,T];W^{k,p}(\R^d))$.

By the Sobelev embedding \eqref{eq:lem:comp_operator;Sob_emb} and density of $\ms{S}(\R^d)$ in $W^{1,p}(\R^d)$ it follows from $v \in C([0,T];W^{1,p}(\R^d))$ that $v \in C([0,T];C_0(\R^d))$, where $C_0(\R^d)$ is the space of continuous functions on $\R^d$ that vanish at infinity.
By a compactness argument we get $\lim_{|x| \to \infty}v(t,x)=0$ uniformly in $t \in [0,T]$, so that
$$
\lim_{|x| \to \infty}\psi'(v(t,x)) = \psi'(0)
$$
uniformly in $t \in [0,T]$.
In light of Lemma~\ref{lem:W1p_operator}, the operator-theoretic conditions of \cite[Section~6.1,pg.~212]{Lu95} are thus satisfied by the linear operator family $\{A_k(t)\}_{t \in [0,T]} \subset \ms{L}(D_k,X_k)$ given by 
$$
A_k(t):D_k \to X_k,\quad \: u \mapsto \psi'(v(t))\Delta u + \psi''(v(t))\nabla v(t) \cdot \nabla u.
$$
As in the proof of Theorem~\ref{thm:decay_at_infty_sol}, it can be shown that $v$ coincides with the unique strict solution  $u$ of
\begin{equation}\label{eq:thm:decay_at_infty_sol_W1p;eq_k}
u'(t) = A_k(t)u(t)+g(t),\:\: 0 < t \leq T; \quad u(0)=v_0.    
\end{equation}
As $g \in C^{\alpha}([0,T];X_{k+1})$, $v_0 \in D_{k+1}$ and $\alpha \in (0,\frac{1}{2}]$, we have 
$$
A_k(0)v_0+g(0) \in X_{k+1} = [X_k,D_k]_{\frac{1}{2}} \hra [X_k,D_k]_{\alpha} \hra (X_k,D_k)_{\alpha,\infty}.
$$
Therefore, by \cite[Corollary~6.1.6(iv)]{Lu95} we get that
\begin{equation*}\label{eq:thm:decay_at_infty_sol_W1p;mixed_der}
v \in C^{1+\alpha}([0,T];X_k) \cap C^\alpha([0,T];D_k) \hra C^\alpha([0,T];D_k) = C^\alpha([0,T];W^{k+2,p}_\kappa(\R^d)). 
\end{equation*}
As above, it follows that the operator-theoretic conditions of \cite[Section~6.1,pg.~212]{Lu95} are satisfied by the linear operator family $\{A_{k+1}(t)\}_{t \in [0,T]} \subset \ms{L}(D_{k+1},X_{k+1})$ given by 
$$
A_{k+1}(t):D_{k+1} \to X_{k+1},\quad \: u \mapsto \psi'(v(t))\Delta u + \psi''(v(t))\nabla v(t) \cdot \nabla u.
$$
An application of \cite[Corollary~6.1.6 and Proposition~6.2.2]{Lu95} gives that there exists a unique strict solution $u \in C^1([0,T];X_{k+1}) \cap C([0,T];D_{k+1})$ of 
\begin{equation*}
u'(t) = A_{k+1}(t)u(t)+g(t),\:\: 0 < t \leq T; \quad u(0)=v_0.    
\end{equation*}
Since $X_{k+1} \hra X_k$, $D_{k+1} \hra D_k$ and $A_{k+1}(t)_{|D_k}=A_k(t)$ for every $t \in [0,T]$, it follows that this $u$ is a strict solution of \eqref{eq:thm:decay_at_infty_sol_W1p;eq_k} as well. By uniqueness, $v=u \in C^1([0,T];X_{k+1}) \cap C([0,T];D_{k+1})$.
\end{proof}

\subsection{Smooth,  non-degenerate approximations of $\phi$}

In the following lemma we construct smooth,  non-degenerate approximations of $\phi$.

\begin{lemma}\label{lem:approx_phi}
Let $I \subset \R$ be an open interval with $0 \in I$ and let $\phi:I \to \R$ be a maximal monotone function with $\phi \in W^{1,\infty}_{\loc}(I)$ and $\phi(0)=0$. 
Then there exist sequences $(\phi_k)_{k \in \N} \subset C^\infty(\R)$ and $(c_k)_{k \in \N} \subset (0,\infty)$ with $\phi_k(0)=0$, $\phi'_k \in BC^\infty(\R)$ and $\phi'_k \geq c_k$ for all $k \in \N$ and such that $\phi_k \to \phi$ as $k \to \infty$ in the sense of Proposition~\ref{thm:conv_m-acrr_A_phi} and $\phi'_k \to \phi'$ as $k \to \infty$ in $(L^\infty_\loc(I),\sigma(L^\infty_\loc(I),L^1_c(I)))$.
\end{lemma}
\begin{proof}     

It will be convenient to write $D:= \phi'$.
Let $\eta \in C^\infty_c(\R)$ with $\eta \geq 0$ and $\int \eta = 1$ and set $\eta_k :=2^{k}\eta(2^k\,\cdot\,)$ for $k \in \N$.
Let  $(I_k)_{k \in \N}$ be a sequence of open intervals of the form $I_k=(a_k,b_k)$ for some $-\infty< a_k < 0 < b_k < \infty$ with $a_k \searrow \inf(I)$ and $b_k \nearrow \sup(I)$ as $k \to \infty$.
Defining
\begin{equation*}
    \tilde{D}_k := 2^{-k} + \one_{I_k}D,
\end{equation*}
we have $\tilde{D}_k \in L^\infty(\R)$ with $\tilde{D}_k \geq 2^{-k}$ and $\tilde{D}_k \to D$ in $L^\infty_\loc(I)$ as $k \to \infty$.
Setting $D_k := \tilde{D}_k * \eta_k$ we have $D_k \in BC^\infty(\R)$ with $D_k \geq 2^{-k}$ and $D_k \to D$ in $(L^\infty_\loc(I),\sigma(L^\infty_\loc(I),L^1_c(I)))$ as $k \to \infty$.
We claim that
$$
\phi_k(r) := \int_{0}^{r}D_k(z)\,\d z, \qquad r \in \R, k \in \N,
$$
is as desired.
To prove this, it remains to show that $\phi_k \to \phi$ as $k \to \infty$ in the sense of Proposition~\ref{thm:conv_m-acrr_A_phi}, that is, $\phi_k \to \phi$ a.e.\ on $I$ and $\phi_k \to \infty$ a.e.\ on $\R \setminus I$ as $k \to \infty$.

Since $D=\lim_{k \to \infty}D_k$ in $(L^\infty_\loc(I,\sigma(L^\infty_\loc(I,L^1_c(I)))$ it follows that $\phi(r) = \lim_{k \to \infty}\phi_k(r)$ for all $r \in I$.
To treat the convergence on $\R \setminus I$, let $r \in \R \setminus I$.
Without loss of generality we assume $r \geq 1$. Fix $k_0 \in \N$ and let $k \in \N$, $k \geq k_0$. Then
$$
\tilde{D}_{k} \geq -2^{-k}+\tilde{D}_{k_0},
$$
so that
$$
D_k \geq -2^{-k}+\tilde{D}_{k_0}*\eta_k \stackrel{k \to \infty}{\longrightarrow} \tilde{D}_{k_0}  \quad \text{in} \quad L^1_\loc(\R).
$$
Therefore,
$$
\liminf_{k \to \infty}\phi_k(r) \geq \int_{0}^{r}\tilde{D}_{k_0}(z)\,\d z 
= \int_{0}^{b_{k_0}}D(z)\,\d z = \phi(b_{k_0}) \stackrel{k_0 \to \infty}{\longrightarrow} \infty.
$$
This shows that $\lim_{k \to \infty}\phi_k(r) = \infty$, as required.
\end{proof}

\subsection{Proof of Proposition~\ref{thm:approx}}\label{subsec:proof_thm:approx}

\begin{proof}[Proof of Proposition~\ref{thm:approx}]
Let $(\phi_\ell)_{\ell \in \N}$ be as in Lemma~\ref{lem:approx_phi}, let $(f_\ell)_{\ell \in \N} \subset C^\infty([0,T];\ms{S}(\R^d))$ be such that $f=\lim_{\ell \to \infty}f_\ell$ in $L^1([0,T] \times \R^d)$ and let $(u_{0,\ell})_{\ell \in \N} \subset \ms{S}(\R^d)$ be such that $u_0 = \lim_{\ell \to \infty}u_{0,\ell}$ in $L^1(\R^d)$. 
Then \ref{it:thm:approx;phi_reg} and \ref{it:thm:approx;phi_conv} are satisfied.

Let $\beta \in (0,1)$ and note that the conditions of Proposition~\ref{pro_LSU} are satisfied with $\psi = \phi_\ell$, $g=f_\ell$ and $v_0=u_{0,\ell}$.
For each $\ell \in \N$ there thus exists a unique classical solution $u_\ell$ of \eqref{eq:approx_eq} and $u_\ell \in BC^{1+\frac{\beta}{2},2+\beta}([0,T]\times \R^d)$.
In particular, \ref{it:thm:approx;class_sol} is satisfied.

Next let $\kappa \in (d-1,d)$ and pick $p \in (d,\infty)$ such that $\kappa < d(1-\frac{1}{p})$ and note that the conditions of Theorem~\ref{thm:decay_at_infty_sol_W1p} are satisfied for $\alpha = \frac{1}{2}$ and any $k \in \N$ with $\psi = \phi_\ell$, $g=f_\ell$, $v_0=u_{0,\ell}$ and $v=u_\ell$.
As a consequence, we obtain that, for every $k \geq 1$,
\begin{align*}
u_\ell \in &\: C^1([0,T];W^{\ell,p}_\kappa(\R^d)) \cap C([0,T];W^{k+2,p}_\kappa(\R^d))   \\
&\hra C^1([0,T];W^{k,p}_\kappa(\R^d))
\hra C^1([0,T];BC^{k-1}_\kappa(\R^d)),
\end{align*}
where the last embedding follows from Proposition~\ref{prop:weighted_Sob_emb}.
This means that $u_k \in C^1([0,T];BC^\infty(\R^d))$ with
$$
|\p_t^j\p_x^\alpha u_k(x)| \lesssim_{k,\kappa,|\alpha|} (1+|x|)^{-\kappa}
$$
for all $x \in \R^d$, $j \in \{0,1\}$, $\alpha \in \N^d$ and $\kappa \in (d-1,d)$,
from which \ref{it:thm:approx;decay} follows by differentiating the equation 
$$
\p_t u_k = \Delta \phi_k(u_k)+f_k
$$
with respect to $t$.
From this it follows that the conditions of Proposition~\ref{lem:classical_implies_mild} are satisfied with $\psi = \phi_\ell$, $g=f_\ell$, $v_0=u_{0,\ell}$ and $v=u_\ell$; indeed, as $\kappa > d-1$, the decay estimates
$$
|\nabla_x u_\ell(t,x)| \lesssim_\ell (1+|x|)^{-\kappa} \quad \text{and} \quad |\p_t\nabla_x u_\ell(t,x)| \lesssim_\ell (1+|x|)^{-\kappa}, \qquad x \in \R^d,
$$
imply that
\begin{align*}
\lim_{R \to \infty}\int_{\p B_R(0)}|\nabla u_\ell(t)|\d\sigma &= \lim_{R \to \infty}\int_{\p B_R(0)}|\p_t\nabla u_\ell(t)|\d\sigma = 0.
\end{align*}
We thus obtain \ref{it:thm:approx;mild_sol_L1}.

Finally, it follows from a combination of Proposition~\ref{ex:thm:conv_m-acrr_A_phi} and Corollary~\ref{cor:thm:nonlin_Trotter-Kato_A_phi} that $u=\lim_{\ell \to \infty} u_\ell$ in $C([0,T];L^1(\R^d))$.
So \ref{it:thm:approx;conv_f_u0_u} is satisfied as well.
\end{proof}

\subsection{A comparison principle and non-negativity of solutions}\label{sec:comp}

\begin{theorem}\label{thm:comparison}
Let $I \subset \R$ be an open interval with $0 \in I$ and let $\phi:I \to \R$ be a maximal monotone function with $\phi \in W^{1,\infty}_{\loc}(I)$ and $\phi(0)=0$.
Assume, in case $d \geq 3$, that there exists $\alpha \geq \frac{d-2}{d}$ such that \eqref{eq:ex:thm:conv_m-acrr_A_phi;suff} holds. 

Let $u_0, \tilde{u}_0 \in L^1(\R^d;\overline{I})$, let $f, \tilde{f} \in L^1([0,T]\times\R^d)$ and let $u, \tilde{u} \in C(0,T];L^1(\R^d;\overline{I}))$ be the unique mild solutions of \eqref{eq:GPME_f_lin} and
\begin{equation}\label{eq:GPME_f_lin_tilde}
\begin{cases}
\p_t \tilde{u} = \Delta \phi(\tilde{u})+\tilde{f} &\text{in}\:\:(0,T) \times \R^d, \\
\tilde{u}(0)=\tilde{u}_{0} &\text{on}\:\:\R^d,    
\end{cases}
\end{equation}
respectively.
Then, for all $t \in (0,T]$,
\begin{align}
\nrm{(u(t)-\tilde{u}(t))_{+}}_{L^1(\R^d)}
& \leq \nrm{(u_0-\tilde{u}_0)_+}_{L^1(\R^d)}
  +  \nrm{(f-\tilde{f})_+ \sign^+_0(u-\tilde{u})}_{L^1((0,t) \times \R^d)} 
\label{eq:thm:comparison;sign+}  \\
& \leq \nrm{(u_0-\tilde{u}_0)_+}_{L^1(\R^d)}
  +  \nrm{(f-\tilde{f})_+}_{L^1((0,t) \times \R^d)}
\label{eq:thm:comparison}
\end{align}
As a consequence,
\begin{equation}\label{eq:thm:comparison;L1_contr}
\nrm{u(t)-\tilde{u}(t)}_{L^1(\R^d)} 
\leq \nrm{u_0-\tilde{u}_0}_{L^1(\R^d)}
  +  \nrm{f-\tilde{f}}_{L^1((0,t) \times \R^d)}.
\end{equation}
\end{theorem}
\begin{proof}
Inspection of the proof of Proposition~\ref{thm:approx} shows that the approximations for \eqref{eq:GPME_f_lin} and \eqref{eq:GPME_f_lin_tilde} can be carried out with a common approximation of $\phi$.
Pick $(\phi_k)_{k \in \N} \subset C^\infty(\R)$,  $(f_k)_{k \in \N}, (\tilde{f}_k)_{k \in \N} \subset L^1([0,T] \times \R^d)$, $(u_{0,k})_{k \in \N}, (\tilde{u}_{0,k})_{k \in \N} \subset L^1(\R^d)$ and $(u_k)_{k \in \N}, (\tilde{u}_k)_{k \in \N} \subset C^{1,2}([0,T] \times \R^d)$ accordingly.

As a consequence of \ref{it:thm:approx;decay}, 
$\int_{\p B_R(0)}|\nabla u_k(t)|\d\sigma \to 0$ and $\int_{\p B_R(0)}|\nabla \tilde{u}_k(t)|\d\sigma \to 0$ for $R \to \infty$.
In view of \ref{it:thm:approx;phi_reg} and \ref{it:thm:approx;class_sol},  the inequality \eqref{eq:lem:L1-contraction;+}
from Lemma~\ref{lem:L1-contraction} thus yields that, for every $t \in (0,T]$,
\begin{align*}
\nrm{(u_k(t)-\tilde{u}_k(t))_+}_{L^1(\R^d)} 
&\leq \nrm{(u_k(0)-\tilde{u}_k(0))_+}_{L^1(\R^d)} \\
&\quad + \nrm{(f_k-\tilde{f}_k)_+\sign^+_0(u-\tilde{u})}_{L^1([0,t] \times \R^d)}.    
\end{align*}
Taking the limit for $k \to \infty$ we arrive at \eqref{eq:thm:comparison;sign+} thanks to the convergence from~\ref{it:thm:approx;conv_f_u0_u}. 
\end{proof}

\begin{corollary}\label{cor:thm:comparison}
Let $\phi:I \to \R$ be as in Theorem~\ref{thm:comparison} and let $f,\tilde{f}\colon\overline{I} \to \R$ be Lipschitz functions with $f(0)=0$ and $\tilde{f}(0)=0$. 
Let $u_0\, \tilde{u_0} \in L^1(\R^d;\overline{I} )$ and let $u, \tilde{u} \in C([0,T];L^1(\R^d;\overline{I} ))$ be the mild solutions of \eqref{eq:GPME2} and
\begin{equation}\label{eq:GPME2;tilde}
\begin{cases}
\p_t \tilde{u} = \Delta \phi(\tilde{u})+ \tilde{f}(\tilde{u}) &\text{in}\:\:(0,T) \times \R^d, \\
\tilde{u}(0)=\tilde{u}_{0} &\text{on}\:\:\R^d,    
\end{cases}
\end{equation}
respectively.
Then, for all $t \in (0,T]$,
\begin{equation}\label{eq:cor:thm:comparison}
\nrm{(u(t)-\tilde{u}(t))_{+}}_{L^1(\R^d)}
\leq e^{Lt}\left(\nrm{(u_0-\tilde{u}_0)_+}_{L^1(\R^d)}
  +  \nrm{(f(u)-\tilde{f}(u))_+}_{L^1((0,t) \times \R^d)}\right),
\end{equation}
where $L=[\tilde{f}]_{\mrm{Lip}}$.
\end{corollary}

\begin{proof}
First recall that $f(u), \tilde{f}(\tilde{u}) \in L^1((0,T) \times \R^d)$, see the proof of Theorem~\ref{thm:wellposedness;Lip_source_term_f}.
So we can apply Theorem~\ref{thm:comparison} with $f(u)$ and $\tilde{f}(\tilde{u})$ in place of $f$ and $\tilde{f}$, respectively, which through the estimate \eqref{eq:thm:comparison;sign+} gives that $$
\nrm{(u(t)-\tilde{u}(t))_{+}}_{L^1(\R^d)}
\leq \nrm{(u_0-\tilde{u}_0)_+}_{L^1(\R^d)}
  +  \nrm{(f(u)-\tilde{f}(\tilde{u}))_+ \sign^+_0(u-\tilde{u})}_{L^1((0,t) \times \R^d)}. 
$$
Since
\begin{align*}
(f(u)-\tilde{f}(\tilde{u}))_+ \sign^+_0(u-\tilde{u})
&= (f(u)-\tilde{f}(u))_+ \sign^+_0(u-\tilde{u}) \\
&\quad + (\tilde{f}(u)-\tilde{f}(\tilde{u}))_+ \sign^+_0(u-\tilde{u}) \\
&\leq (f(u)-\tilde{f}(u))_+ + |\tilde{f}(u)-\tilde{f}(\tilde{u}))|\sign^+_0(u-\tilde{u}) \\
&\leq (f(u)-\tilde{f}(u))_+ + L|u-\tilde{u}|\sign^+_0(u-\tilde{u}) \\
&= (f(u)-\tilde{f}(u))_+ + L(u-\tilde{u})_{+},
\end{align*}
it follows that
\begin{align*}
\nrm{(u(t)-\tilde{u}(t))_{+}}_{L^1(\R^d)}
&\leq \nrm{(u_0-\tilde{u}_0)_+}_{L^1(\R^d)}
+ \nrm{(f(u)-\tilde{f}(u))_+}_{L^1((0,t) \times \R^d)} \\
&\quad + \int_{0}^{t} L\nrm{(u(s)-\tilde{u}(s))_{+}}_{L^1(\R^d)}\d s.
\end{align*}
By the integral form of Gronwall's inequality we thus obtain \eqref{eq:cor:thm:comparison}.
\end{proof}

\begin{corollary}\label{cor:cor:thm:comparison;positivity}
Let the notations and assumptions be as in Theorem~\ref{thm:wellposedness;Lip_source_term_f} and in addition assume that $\phi \in W^{1,\infty}(I)$ and that, in case $d \geq 3$, there exists $\alpha \geq \frac{d-2}{d}$ such that \eqref{eq:ex:thm:conv_m-acrr_A_phi;suff} holds.
If $u_0 \geq 0$, then $u \geq 0$.
\end{corollary}

A question that arises is whether the above result could be obtained from the abstract theory of $m$-accretive operators on Banach lattices.
Related results in the setting of maximal monotone operators on Hilbert lattices were obtained in \cite{CG01}.

\begin{proof}
Note that $\tilde{u}=0$ is the mild solution of \eqref{eq:GPME2;tilde} for $\tilde{u}_0=0$ and $\tilde{f}=0$. 
Applying Corollary~\ref{cor:cor:thm:comparison;positivity} with the roles of $u$ and $\tilde{u}$ interchanged, we obtain that, for all $t \in (0,T]$,
\begin{align*}
\nrm{(u(t))_{-}}_{L^1(\R^d)}
&= \nrm{(\tilde{u}(t)-u(t))_{+}}_{L^1(\R^d)} \\
&\leq e^{Lt}\left(\nrm{(\tilde{u}_0-u_0)_+}_{L^1(\R^d)}
  +  \nrm{(\tilde{f}(\tilde{u})-f(\tilde{u})_+}_{L^1((0,t) \times \R^d)}\right) \\
&= e^{Lt}\left(\nrm{(u_0)_-}_{L^1(\R^d)}
  +  \nrm{(f(0))_-}_{L^1((0,t) \times \R^d)}\right)=0,   
\end{align*}
where we used that $f(0)=0 \geq 0$ and $u_0 \geq 0$ in the last equality. 
This implies that $(u(t))_-=0$, that is, $u(t) \geq 0$.
\end{proof}

\section{Regularity}\label{sec:reg}

In this section we address the Sobolev regularity of solutions and prove Theorem \ref{thm:regularity}. To this end we first derive the kinetic formulation for \eqref{eq:GPME} and then apply Fourier analytic techniques to deduce the resularity results.

\subsection{Kinetic formulation}

We aim to rigorously derive the kinetic formulation for \eqref{eq:GPME}.
To this end let $u:[0,T]\times\R^d\to\R$. We introduce the corresponding kinetic function 
\begin{align}
\chi(t,x,v;u) :&= \begin{cases}
+1, & 0 < v < u(t,x),\\
-1, & u(t,x) < v < 0,\\
0, & \text{otherwise}, 
\end{cases}      \label{eq:def_chi}  \\
&= \one_{\{0<v<u(t,x)\}}-\one_{\{u(t,x)<v<0\}}, \nonumber
\end{align}
for $(t,x)\in[0,T]\times\R^d, v\in \R.$ 
This kinetic function satisfies a linear parabolic equation if $u$ is the mild solution of \eqref{eq:GPME_f_lin}.

\begin{proposition}\label{prop:mild_is_kinetic}
Let $I \subset \R$ be an open interval with $0 \in I$ and let $\phi:I \to \R$ be a maximal monotone function with $\phi \in W^{1,\infty}_{\loc}(I)$ and $\phi(0)=0$.
Assume, in case $d \geq 3$, that there exists $\alpha \geq \frac{d-2}{d}$ such that \eqref{eq:ex:thm:conv_m-acrr_A_phi;suff} holds.

Let $u_0 \in L^1(\R^d;\overline{I})$, $f \in L^1([0,T]\times\R^d)$ and $u \in C(0,T];L^1(\R^d;\overline{I}))$ be the unique mild solution of \eqref{eq:GPME_f_lin}.
Then there exists a positive measure $n \in \Lambda^\infty(I;\ms{M}((0,T)\times \R^d))$ such that
\begin{equation}\label{eq:prop:mild_is_kinetic}
\partial_t \chi(t,x,v;u) = \phi'(v)\Delta_x\chi(t,x,v;u) + \partial_v n(t,x,v)+\delta_{v=u(t,x)}f    
\end{equation}
in the sense of distributions on $(0,T) \times \R^d_x \times I_v$.
Moreover, 
\begin{equation}\label{eq:prop:mild_is_kinetic;u_est_mild_sol}
\nrm{u}_{C([0,T];L^1(\R^d))} \leq  \norm{u_0}_{L^1(\R^d)}+\norm{f}_{L^1([0,T] \times \R^d)} 
\end{equation}
and
\begin{equation}\label{eq:prop:mild_is_kinetic;measure}
\norm{n}_{\Lambda^\infty(I;\ms{M}((0,T)\times \R^d))} \leq \norm{u_0}_{L^1(\R^d)}+\norm{f}_{L^1([0,T] \times \R^d)}.  
\end{equation}
\end{proposition}

Proposition \ref{prop:mild_is_kinetic} is inspired by \cite[Section~2]{CP03} and \cite[Lemma~A.2]{GST20}.
Different from these works we link the mild solution concept to kinetic solutions.
However, besides the kinetic equation itself \eqref{eq:prop:mild_is_kinetic}/\cite[(2.16)]{CP03}, our kinetic formulation is quite different from \cite[Definition~2.2]{CP03}.
Indeed, motivated by applying it to prove Sobolev regularity, the only property of the measure $n$ that is relevant to us is that $n \in \Lambda^\infty(I;\ms{M}((0,T)\times \R^d))$.
This property is not included in \cite[Definition~2.2]{CP03}, but could be derived from it as shown in \cite[Lemma~A.2]{GST20}.

In the proof of Proposition \ref{prop:mild_is_kinetic} we will use Theorem~\ref{thm:approx} to approximate the equation \eqref{eq:GPME_f_lin} by smooth, non-degenerate versions.
The approximate solutions are  sufficiently regular such that we can carry out the necessary computations leading to the following lemma.
In the proof we use the following observation that also lies at the core of velocity averaging.
Let $H\in C^1(\R)$ be arbitrary with $H(0)=0$. Then, if $u$ takes its values in $\overline{I}$ a.e., $I\subset\R$ an interval, we have
\begin{equation}\label{eq:averaging_formula}
H(u(t,x)) = \int_I H'(v)\chi(t,x,v;u)\d v.    
\end{equation}
Indeed, this follows from 
\begin{align*}
\int_I H'(v)\chi(t,x,v;u) \d v
&= \int_I H'(v)(\one_{\{0<v<u\}}-\one_{\{u<v<0\}} )\d v \\
&= \begin{cases}
    0, & u(t,x)=0 \\
    \int_{0}^{u(t,x)}H'(v)\d v, & u(t,x)>0,\\
    -\int_{u(t,x)}^{0}H'(v)\d v, & u(t,x)<0,
\end{cases} \\
&= \begin{cases}
    0=H(u(t,x)), & u(t,x)=0 \\
    H(u(t,x)), & u(t,x)>0,\\
    --H(u(t,x))=H(u(t,x)), & u(t,x)<0,
\end{cases} \\
&= H(u(t,x)).
\end{align*}

\begin{lemma}\label{lemma:GPME;psi_v}
Let $\psi:\R \to \R$ be a $C^2$-function with $\psi(0)=0$, $\psi', \psi'' \in L^\infty(\R)$ and $\psi'(r)>0$ for all $r \in \R$.
Let $g \in L^1([0,T] \times \R^d)$, $w_0 \in L^1(\R^d)$ and $w \in C^{1,2}([0,T] \times \R^d)$ be a classical solution of 
\begin{equation*}
\begin{cases}
\p_t w = \Delta \psi(w)+g &\text{in}\:\:(0,T) \times \R^d, \\
w(0)=w_{0} &\text{on}\:\:\R^d.    
\end{cases}
\end{equation*}
with $\int_{\p B_R(0)}|\nabla w|\d\sigma \to 0$ for $R \to \infty$.
Then there exists a positive measure $n \in \Lambda^\infty(\R;\ms{M}((0,T)\times \R^d))$ such that
\begin{equation}\label{eq:prop:mild_is_kinetic;psi_v}
\partial_t \chi(t,x,v;w) = \psi'(v)\Delta_x\chi(t,x,v;w) + \partial_v n(t,x,v)+\delta_{v=w(t,x)}g      
\end{equation}
in the sense of distributions on $(0,T) \times \R^d \times \R$.
Moreover, the measure $n$ satisfies 
\begin{equation}\label{eq:lemma:GPME;psi_v;est;new}
\norm{n}_{\Lambda^\infty(\R;\ms{M}((0,T)\times \R^d))} \leq \norm{w_0}_{L^1(\R^d)}+\norm{g}_{L^1([0,T] \times \R^d)}.    
\end{equation}
\end{lemma}

This lemma is to a large extend based on computations in \cite[Section~2]{CP03}.
The main difference is that we omit the entropy formulation and directly use the kinetic formulation.
Moreover, for the kinetic formulation the same differences that were mentioned after Proposition \ref{prop:mild_is_kinetic} apply.

\begin{proof}
We will show that the kinetic equation \eqref{eq:prop:mild_is_kinetic;psi_v} is satisfied for the positive measure $n \in M(\R;L^1([0,T]\times \R^d))$ induced by
$$
C_c(\R) \longrightarrow L^1([0,T]\times \R^d),\: \eta \mapsto \eta(w)\psi'(w)|\nabla w|^2
$$
and that this measure satisfies the norm estimate 
\begin{equation}\label{eq:lemma:GPME;psi_v;est}
\norm{n(\eta)}_{L^1([0,T]\times \R^d)} \leq \norm{\eta}_{L^1(\R)}(\norm{w_0}_{L^1(\R^d)}+\norm{g}_{L^1([0,T] \times \R^d)})    
\end{equation}
for all $\eta \in C_c(\R)$.
Moreover, we will subsequently derive \eqref{eq:lemma:GPME;psi_v;est;new} from \eqref{eq:lemma:GPME;psi_v;est}.

\emph{Step 1:} For every $S \in C^2(\R)$ the following equation is satisfied:
\begin{equation}\label{eq:lemma:GPME;psi_v;entropy1}
\p_{t}S(w)-\nabla \cdot(S'(w)\psi'(w)\nabla w) = - S''(w)\psi'(w)|\nabla w|^2 + S'(w)g.
\end{equation}

Indeed, as $S'(w) \p_t w  = \p_{t}S(w)$ and
\begin{align*}
&S'(w)\Delta \psi(w)
= S'(w)\nabla \cdot(\psi'(w)\nabla w) \\
=& S''(w)\psi'(w)|\nabla w|^2 + S'(w)\nabla \cdot(\psi'(w)\nabla w) - S''(w)\psi'(w)|\nabla w|^2 \\
=& \left(\left(S''(w)\nabla  w\right) \psi'(w) \nabla w + S'(w)\nabla \cdot(\psi'(w)\nabla w)\right) 
 - S''(w)\psi'(w)|\nabla w|^2 \\
=& \nabla \cdot\left(S'(w)\psi'(w)\nabla w\right) - S''(w)\psi'(w)|\nabla w|^2,
\end{align*}
multiplying the equation $\p_t w = \Delta \psi(w)+g$ with $S'(w)$ and subsequently rearranging the terms yields \eqref{eq:lemma:GPME;psi_v;entropy1}.

\emph{Step 2:} We now prove the estimate \eqref{eq:lemma:GPME;psi_v;est}.
Let $\eta \in C_c(\R)^+$ and note that the function 
\begin{equation*}
    S(v) := \int_0^v\int_0^\zeta \eta(s)\d s \d \zeta
\end{equation*}
satisfies $S \in C^2(\R)$, $S''=\eta \geq 0$, $S \geq 0$, $S(0)=0$ 
and $\norm{S'}_{L^\infty(\R)} \leq \norm{\eta}_{L^1(\R)}$.
Using \eqref{eq:lemma:GPME;psi_v;entropy1} and invoking Fubini's theorem, we find that
\begin{align*}
\int_0^T\int_{B_R(0)}n(S'')  
&=\int_0^T\int_{B_R(0)}\left(-\partial_t S(w) + \nabla \cdot(S'(w)\psi'(w)\nabla w)+S'(w)g \right)
\end{align*}
for every $R>0$.
As $S \geq 0$ and $S(0)=0$, the first term on the right hand side can be estimated by
\begin{equation*}
\int_{B_R(0)}\int_0^T-\partial_t S(w)  \leq
\int_{B_R(0)}S(w_0)  \leq |B_R(0)|\norm{S'}_{L^\infty(\R)}\norm{w_0}_{L^1(\R^d)},
\end{equation*}
and the third term by
\begin{equation*}
\int_0^T\int_{B_{R}(0)} S'(w)g 
\leq \norm{S'}_{L^\infty(\R)}\norm{g}_{L^1([0,T] \times \R^d)}.  
\end{equation*}
By the Gau\ss-Green theorem we have
\begin{align*}
\left|\int_{B_R(0)}\nabla \cdot(S'(w)\psi'(w)\nabla w) \right| 
&= \left| \int_{\p B_R(0)}S'(w)\psi'(w)\nabla w \cdot \nu \,\d\sigma \right| \\
&\leq \norm{S'}_{L^\infty(\R)}\norm{\psi'}_{L^\infty(\R)}\int_{\p B_R(0)}|\nabla w|\d\sigma. 
\end{align*}
As $\int_{\p B_R(0)}|\nabla w|\d\sigma \to 0$ for $R \to \infty$ by assumption and $n(S'') \geq 0$ by positivity of $n$ and $S'' \geq 0$, by monotone convergence it follows that
\begin{equation*}
\int_0^T\int_{\R^d}n(S'')\leq
\norm{S'}_{L^\infty(\R)}\left(\norm{w_0}_{L^1(\R^d)}+\norm{g}_{L^1([0,T] \times \R^d)}\right).    
\end{equation*}
Recalling that $S''= \eta$ and $\norm{S'}_{L^\infty(\R)} \leq \norm{\eta}_{L^1(\R)}$, we arrive at the desired estimate \eqref{eq:lemma:GPME;psi_v;est}.

\emph{Step 3:} Next, we show that the kinetic formulation \eqref{eq:prop:mild_is_kinetic;psi_v} holds.
Fix test functions $\zeta \in \ms{D}(\R)$ and $\varphi \in \ms{D}((0,T) \times \R^d)$.
Let $S \in C^2(\R)$ be such that $S(0)=0$ and $S'= \zeta$ and let $G \in C^1(\R)$ be such that $G(0)=0$ and $G'=S'\psi'$.
Then, we have
$$
\nabla \cdot(S'(w)\psi'(w)\nabla w) = \nabla \cdot(\nabla G(w)) = \Delta G(w),
$$
and therefore, \eqref{eq:lemma:GPME;psi_v;entropy1} can be rewritten as
\begin{equation}\label{eq:lemma:GPME;psi_v;entropy2}
\p_{t}S(w)-\Delta G(w) = - n(S'') + S'(w)g.     
\end{equation}
Using integration by parts, Fubini's Theorem  and \eqref{eq:averaging_formula} applied to $S$ and $G$, we have 
\begin{align*}
\int_{(0,T) \times \R^d}\p_{t}\ S
& (w(t,x))\varphi(t,x)\d(t,x) \\
&= -\int_{(0,T) \times \R^d}\ S(w(t,x))\p_{t}\varphi(t,x)\d(t,x) \\
&= -\int_{(0,T) \times \R^d}\int_{\R}\zeta(v)\chi(t,x,v;w)\d v\, \p_{t}\varphi(t,x)\d(t,x) \\
&= -\int_{(0,T) \times \R^d \times \R}\chi(t,x,v;w) \p_{t}\varphi(t,x)\zeta(v)\,\d(t,x,v), 
\end{align*}
and
\begin{align*}
\int_{(0,T) \times \R^d}\Delta &G(w(t,x))(w(t,x))\varphi(t,x)\d(t,x) \\
&= \int_{(0,T) \times \R^d}G(w(t,x))\Delta\varphi(t,x)\d(t,x) \\
&= \int_{(0,T) \times \R^d}\int_{\R}\zeta(v)\psi'(v)\chi(t,x,v;w)\d v\, \Delta\varphi(t,x)\d(t,x) \\
&= \int_{(0,T) \times \R^d \times \R}\psi'(v)\chi(t,x,v;w) \Delta\varphi(t,x)\zeta(v)\,\d(t,x,v), \\    
\end{align*}
respectively.
Observing that by the definition of the distributional derivative $n(\zeta')=-\p_v n(\zeta)$  and combining the above two identities with~\eqref{eq:lemma:GPME;psi_v;entropy2}, we conclude that \eqref{eq:prop:mild_is_kinetic;psi_v} holds in the sense of distributions on $(0,T) \times \R^d \times \R$.

\emph{Step 4:} Finally, we prove that the measure $n$ belongs to $\Lambda^\infty(\R;\ms{M}((0,T)\times \R^d))$ with norm estimate \eqref{eq:lemma:GPME;psi_v;est;new}.
Recall that $L^1([0,T] \times \R^d) \subset \ms{M}((0,T)\times \R^d)$ isometrically and that, by the Riesz representation theorem (see e.g.\ \cite[Theorem~C.18]{Co96}),
\begin{equation}\label{eq:Riesz_repr_thm}
\ms{M}((0,T)\times \R^d) 
= [C_0((0,T)\times \R^d)]^*.    \nonumber
\end{equation}
It follows from \cite[Proposition~2.28]{Pi16} that
\begin{equation}\label{eq:Lambda_Pisier_dual_space}
\Lambda^\infty(\R;\ms{M}((0,T)\times \R^d)) =
[L^1(\R;C_0((0,T)\times \R^d))]^*.
\end{equation} 
In combination with \eqref{eq:lemma:GPME;psi_v;est}, this means that we can extend and view $n$ in the natural way as an element of $\Lambda^\infty(\R;\ms{M}((0,T)\times \R^d))$ satisfying the norm estimate \eqref{eq:lemma:GPME;psi_v;est;new}.
\end{proof}

The next lemma provides a temporal-pointwise $L^1$-isometry between space-time functions and their corresponding kinetic functions.

\begin{lemma}\label{lem:chi_cont_est}
Let $u,w \in C([0,T];L^1(\R^d))$. Then, for each $t \in [0,T]$, the corresponding kinetic functions satisfy
\begin{equation}\label{eq:lem:chi_cont_est;diff}
\norm{\chi(t,\,\cdot\,,\,\cdot\,;u)-\chi(t,\,\cdot\,,\,\cdot\,;w)}_{L^1(\R^{d+1})} 
= \norm{u(t)-w(t)}_{L^1(\R^d)}.
\end{equation}
In particular, taking $w=0$,
\begin{equation}\label{eq:lem:chi_cont_est}
\norm{\chi(t,\,\cdot\,,\,\cdot\,;u)}_{L^1(\R^{d+1})} 
= \norm{u(t)}_{L^1(\R^d)}.
\end{equation}
\end{lemma}
\begin{proof}
Using the observation that
\begin{align*}
|\chi(t,x,v;u)&-\chi(t,x,v;w)| = \one_{\{w(t,x)<v<u(t,x)\}} + \one_{\{u(t,x)<v<w(t,x)\}} \quad  \text{a.e.}, 
\end{align*}
we find that
\begin{align*}
\norm{\chi(t,\,\cdot\,,\,\cdot\,;u)-\chi(t,\,\cdot\,,&\,\cdot\,;w)}_{L^1(\R^{d+1})} \\
&= \int_{\R^d}\int_{\R} \one_{\{w(t,x)<v<u(t,x)\}} + \one_{\{u(t,x)<v<w(t,x)\}} \d v \d x \\
&= \int_{\R^d} |u(t,x)-w(t,x)| \d x \\
&= \norm{u(t)-w(t)}_{L^1(\R^d)}.
\end{align*}
\end{proof}

\begin{proof}[Proof of Proposition~\ref{prop:mild_is_kinetic}]
Pick $\kappa \in (d-1,d)$ and let $(\phi_k)_{k \in \N}$, $(f_k)_{k \in \N}$, $(u_{0,k})_{k \in \N}$ and $(u_k)_{k \in \N}$ be as in Theorem~\ref{thm:approx}.
Then, in particular, 
$$
|\nabla_x u_k(t,x)| \lesssim_k (1+|x|)^{-\kappa}, \qquad x \in \R^d,
$$
and thus
\begin{align*}
\lim_{R \to \infty}\int_{\p B_R(0)}|\nabla u_k(t)|\d\sigma  = 0.
\end{align*}
Now note that the conditions of Lemma~\ref{lemma:GPME;psi_v} are satisfied with $\psi=\phi_k$, $v=u_k$, $g=f_k$, $v_0=u_{0,k}$ and $v=u_k$ for each $k \in \N$.
Therefore, there exist  positive measures $n_k \in \Lambda^\infty(\R;\ms{M}((0,T) \times \R^d))$ such that
\begin{equation}\label{eq:prop:mild_is_kinetic;approx_kin_eq}
\partial_t \chi(t,x,v;u_k) = \phi_k'(v)\Delta_x\chi(t,x,v;u_k) + \partial_v n_k(t,x,v)+\delta_{v=u_k(t,x)}f_k      
\end{equation}
in the sense of distributions on $(0,T) \times \R^d_x \times \R$, and
\begin{equation}\label{eq:prop:mild_is_kinetic;est_measures}
\norm{n_k}_{\Lambda^\infty(\R;\ms{M}((0,T) \times \R^d))} \leq \norm{u_{0,k}}_{L^1(\R^d)}+\norm{f_k}_{L^1([0,T] \times \R^d)}.   
\end{equation}

To complete the proof, we will show that \eqref{eq:prop:mild_is_kinetic} can be obtained as the limit of \eqref{eq:prop:mild_is_kinetic;approx_kin_eq} restricted to $I$ for $k \to \infty$ for a suitable positive measure $n \in \Lambda^\infty(I;\ms{M}((0,T) \times \R^d))$ that satisfies 
\eqref{eq:prop:mild_is_kinetic;measure}.

Let us first note that, by Proposition~\ref{cor:thm:nonlin_Trotter-Kato_A_phi}, $u_k \to u$ in $C([0,T];L^1(\R^d))$ as $k \to \infty$.
In view of Lemma~\ref{lem:chi_cont_est} this implies that $\chi(\cdot,\cdot,\cdot;u_k) \to \chi(\cdot,\cdot,\cdot;u)$ in $C([0,T];L^1(\R^{d+1}))$.
It follows that $\partial_t \chi(\cdot,\cdot,\cdot;u) = \lim_{k \to \infty} \partial_t \chi(\cdot,\cdot,\cdot;u_k)$ in $\ms{D}'((0,T) \times \R^{d}\times I)$ after restricting to $I$.

To treat the first term on the right-hand side of \eqref{eq:prop:mild_is_kinetic;approx_kin_eq}, let $\eta \in \ms{D}((0,T) \times \R^d \times I)$.
As $\phi_k' \to \phi'$ as $k \to \infty$ in $(L^\infty_\loc(I),\sigma(L^\infty_\loc(I)),L^1_c(I))))$, $\chi(\cdot,\cdot,\cdot;u_k) \to \chi(\cdot,\cdot,\cdot;u)$ in $C([0,T];L^1(\R^{d+1}))$ and $\Delta_x\eta \in L^1_c(I)$, we have
\begin{align*}
[\phi_k'(v)\Delta_x\chi(t,x,v;u_k)](\eta) &= \int_{(0,T) \times \R^d \times I}\phi_k'(v)\chi(t,x,v;u_k)\Delta_x\eta(t,x,v)  \\
&\stackrel{k \to \infty}{\longrightarrow}
\int_{(0,T) \times \R^d \times I}\phi'(v)\chi(t,x,v;u)\Delta_x\eta(t,x,v) \\
&= [\phi'(v)\Delta_x\chi(t,x,v;u)](\eta). 
\end{align*}
This shows that 
$$
\phi'(v)\Delta_x\chi(t,x,v;u) = \lim_{k \to \infty} \phi_k'(v)\Delta_x\chi(t,x,v;u_k)\quad  \text{in}\ \ms{D}'((0,T) \times \R^d \times I).
$$

To treat the third term on the right-hand side of \eqref{eq:prop:mild_is_kinetic;approx_kin_eq}, let again $\eta \in \ms{D}((0,T) \times \R^d \times I)$.
As
$$
u = \lim_{k \to \infty}u_k \quad \text{in} \quad C([0,T];L^1(\R^d)) \hra L^1((0,T) \times \R^d),
$$
by restricting to a subsequence we may without loss of generality assume that $u = \lim_{k \to \infty}u_k$ pointwise almost everywhere.
By the Lebesgue dominated convergence theorem we get that
$$
[(t,x) \mapsto \eta(t,x,u(t,x))] = \lim_{k \to \infty}[(t,x) \mapsto \eta(t,x,u_k(t,x))]
$$
in $(L^\infty((0,T)\times\R^d),\sigma(L^\infty((0,T)\times\R^d),L^1((0,T)\times\R^d))$ (i.e.\ the weak$^*$-topology on $L^\infty$ as a dual to $L^1$).
Since $f=\lim_{k \to \infty}f_k$ in $L^1((0,T)\times\R^d)$, it follows that
\begin{align*}
\delta_{v=u_k(t,x)}f_k \,(\eta) 
&= \int_{(0,T) \times \R^d}\eta(t,x,u_k(t,x))f_k(t,x)d(t,x) \\
&\stackrel{k \to \infty}{\longrightarrow} \int_{(0,T) \times \R^d}\eta(t,x,u(t,x))f(t,x)d(t,x)    \\
&= \delta_{v=u(t,x)}f \,(\eta).
\end{align*}
So we also have $\delta_{v=u(t,x)}f = \lim_{k \to \infty} \delta_{v=u_k(t,x)}f_k$ in $\ms{D}'((0,T) \times \R^d \times I)$.

It remains to treat the second term on the right-hand side of \eqref{eq:prop:mild_is_kinetic;approx_kin_eq}.
By
\eqref{eq:prop:mild_is_kinetic;est_measures} and restriction from $\R$ to $I$, we can view $(n_k)_{k \in \N}$ as a bounded sequence in $\Lambda^\infty(I;\ms{M}((0,T)\times \R^d))$.
In exactly the same way as \eqref{eq:Lambda_Pisier_dual_space}, we have the dual characterization
\begin{equation*}
\Lambda^\infty(I;\ms{M}((0,T)\times \R^d)) =
[L^1(I;C_0((0,T)\times \R^d))]^*. 
\end{equation*}
Since $C_0((0,T)\times \R^d)$ is separable, it follows from the Banach-Alaoglu theorem (see e.g.\ \cite[Theorem~B.1.7]{HNVW16}) that $(n_k)_{k \in \N}$ has a convergent subsequence in $\Lambda^\infty(I;\ms{M}((0,T)\times \R^d))$ with respect to the weak$^*$-topology. 
Taking this subsequence, we may without loss of generality assume that $(n_k)_{k \in \N}$ has  a weak$^*$-limit in $\Lambda^\infty(I;\ms{M}((0,T)\times \R^d))$. 
Denoting this limit by $n$, we have
\begin{align*}
\norm{n}_{\Lambda^\infty(I;\ms{M}((0,T)\times \R^d))} 
&\leq \liminf_{k \to \infty}\norm{n_k}_{\Lambda^\infty(I;\ms{M}((0,T)\times \R^d))} \\
&\stackrel{\eqref{eq:prop:mild_is_kinetic;est_measures}}{\leq} \liminf_{k \to \infty}(\norm{u_{0,k}}_{L^1(\R^d)}+\norm{f_k}_{L^1([0,T] \times \R^d)}) \\
&= \norm{u_0}_{L^1(\R^d)}+\norm{f}_{L^1([0,T] \times \R^d)}.      
\end{align*}
Furthermore, 
$$
n = \lim_{k \to \infty}n_{k} \quad \text{in} \quad \ms{D}'((0,T) \times \R^d \times I)
$$
and $n$ is positive as every $n_k$ is positive.
\end{proof}

\subsection{Space time Sobolev regularity}\label{subsec:regularity}

The following theorem is  a modification of
\cite[Theorem 1.2]{GS23}  adjusted to our setting.

\begin{theorem}\label{thm:GST20_thm1.1}
Let $I$ be an open interval with $0 \in I$ and $\phi:I \to \R$ be a maximal monotone function with $\phi \in C^1(\R) \cap W^{2,1}_{\loc}(I)$, $\phi(0)=0$ and $|\phi''| = \sign_0 \,\cdot\, \phi''$ on $I \setminus \{0\}$ and let $m \in (1,\infty)$ and $c \in (0,\infty)$ be such that $D=\phi'$ satisfies
\begin{equation}\label{eq:thm:GST20_thm1.1;assump_D}
|D(r)| \geq c|r|^{m-1}, \qquad r \in I.
\end{equation}
Assume, in case $d \geq 3$, that for each compact interval $J \subset I$ there exists a finite constant $C_J>0$ such that
\begin{equation*}
|\phi(r)| \leq C_J |r|^{m}, \qquad r \in J.    
\end{equation*}  
Let $u_0 \in L^1(\R^d;\overline{I})$, $f \in L^1([0,T]\times\R^d))$ and $u \in C(0,T];L^1(\R^d;\overline{I}))$ be the unique mild solution of \eqref{eq:GPME_f_lin}.   

Let $p \in (1,m]$ and define
$$
\kappa_t := \frac{m-p}{p}\frac{1}{m-1},\quad \kappa_x := \frac{p-1}{p}\frac{2}{m-1}.
$$
Let $s \in [0,1]$ and $q \in [1,p]$, and assume that
\begin{equation}\label{eq:assump_s,q}
(-\Delta_x)^{s}\phi(u) \in L^q([0,T]\times\R^d).    
\end{equation}
Then, for all $\sigma_t \in [0,\kappa_t) \cup \{0\}$ and $\sigma_x \in [0,\kappa_x)$, 
$$
u \in W^{\sigma_t,p}(0,T;W^{\sigma_x,p}(\R^d))
$$
with corresponding norm estimate
\begin{equation}\label{eq:thm:GST20_thm1.1;i}
    \nrm{u}_{W^{\sigma_t,p}(0,T;W^{\sigma_x,p}(\R^d))} \lesssim \nrm{u_0}_{L^1(\R^d)}+\nrm{f}_{L^1([0,T] \times \R^d)}+\nrm{(-\Delta_x)^{s}\phi(u)}_{L^q_{t,x}}+1.    
    \end{equation}
\end{theorem}

\begin{remark}\label{rmk:GS23}
We comment on the differences between Theorem \ref{thm:GST20_thm1.1} and  \cite[Theorem 1.2]{GS23}. 
In the latter, porous medium type equations are considered, i.e.\ $\phi$ is a function on $\R$ with degeneracy in $0$ and without singularities. 
Moreover, linear integro-differential operators  are addressed  including as a particular case the fractional Laplacian $(-\Delta)^{\alpha}$. In this case, 
assumption \eqref{eq:assump_s,q} is restricted to the case $s=0$ and $q=1$. 
Finally, for the classical porous medium equation it is shown in \cite{GST20,GS23} that the regularity results are optimal using the Barenblatt solution and scaling arguments. This is out of our scope as source type solutions are unknown for equations of the form 
\eqref{eq:cauchy}.
\end{remark}

\begin{remark}\label{rmk:assump_s;q}
We make some remarks concerning the assumption~\eqref{eq:assump_s,q} which generalizes the hypothesis in \cite{GS23}.
\begin{enumerate}[(i)]
    \item[(i)] For the porous medium equation, \eqref{eq:assump_s,q} is automatically satisfied for $s=0$ and $q=1$, see  \cite[Theorem~1.2]{GS23}.
    \item[(ii)] In our case, we do not know whether \eqref{eq:assump_s,q} is satisfied in the general setting for some $s \in [0,1]$ and $q \in [1,p]$. 
    However, in the special case that $m \geq 2$, we can take $p \geq 2 = q$ and $s=\frac{1}{2}$, in which case \eqref{eq:assump_s,q} is equivalent to
    \begin{equation*}
    \nabla_x \phi(u) \in L^2([0,T] \times \R^d).    
    \end{equation*}
    The latter is for instance satisfied in the setting of Theorem~\ref{thm:energy_est;Lip_source_term_f} and, in particular, in Example \ref{ex:eq:thm:energy_est;Lip_source_term_f;Phi_Lq}. This shows that the hypotheses of Theorem \ref{thm:GST20_thm1.1} are satisfied for Equation \eqref{eq:GPME} in the specific case of the biofilm diffusion coefficient $\phi'(u)=\frac{|u|^b}{(1-|u|)^a}$ in \eqref{eq:bfdiff}. 
\end{enumerate}
\end{remark}

\begin{proof}[Proof of Theorem~\ref{thm:regularity}]
This is an immediate consequence of Theorem \ref{thm:GST20_thm1.1} and Remark \ref{rmk:assump_s;q} (ii).
\end{proof}

The proof of Theorem \ref{thm:GST20_thm1.1} is based on a series of lemmas that are modifications and generalizations of lemmas in \cite{GS23,GST20}. 
The following result is a version of \cite[Lemma~4.4]{GS23} adjusted to our setting.
It provides an estimate for the inhomogeneous Littlewood-Paley blocks $(\overline{\chi}_{l,j})_{(l,j) \in \N_0 \times \N_0}$, which are defined as follows.

Let $(\psi_l)_{l \in \N_0} \in \Phi(\R)$ and $(\varphi_j)_{j \in \N_0} \in \Phi(\R^d)$.
For a tempered distribution $h \in \Schw'(\R \times \R^d)$ we define
$$
h_{l,j} := \psi_{l}\varphi_j * h ,\qquad l,j \in \N_0,
$$
where we use the notation $\psi_{l}\varphi_j = \psi_{l} \otimes \varphi_j = [(t,x) \mapsto \psi_l(t)\varphi_j(x)]$.
In order to avoid confusion, let us remark that we use a different convention than \cite{GS23} regarding the notation of the Littlewood-Paley sequences. Indeed, there is a Fourier transform difference between our sequences $(\psi_l)_{l \in \N_0}$, $(\varphi_j)_{j \in \N_0}$ and the ones in \cite{GS23}.

It will be convenient to use short-hand notation for some of the spaces of functions and measures. We will omit writing the underlying space when it is clear from the context and replace it by a subscript with the corresponding variable name.
For example, we write
$$
\ms{M}(\R \times \R^d \times I) = \ms{M}_{t,x,v}, \qquad \Lambda^\infty(I;\ms{M}(\R \times \R^d)) = \Lambda^\infty_v\ms{M}_{t,x},
$$
and
$$
L^m(\R;W^{\kappa_x,m}(\R^d)) = L^m_t W^{\kappa_x,m}_x.
$$

\begin{lemma}\label{lem:GS23_Lem4.2}
Let $m \in (1,\infty)$ and let $\phi \in C^1(\R) \cap W^{2,1}_{\loc}(I)$ satisfy $D:=\phi' \geq 0$, $|\phi''| = \sign_0 \,\cdot\,\phi''$ on $I \setminus \{0\}$, $D(v) \to \infty$ as $v \to \sup(I)$ and $v \to \inf(I)$  and \eqref{eq:thm:GST20_thm1.1;assump_D} holds true for some $c \in (0,\infty)$.
Let $\chi \in L^\infty([0,T] \times \R^d \times I)$ with $\nrm{\chi}_{L^\infty_{t,x,v}} \leq 1$ be a solution to
\begin{equation}\label{eq:lem:GS23_Lem4.2;kinetic_eq}
\partial_t \chi - D\Delta_x \chi = g + \p_v n
\end{equation}
in the sense of distributions on $\R \times \R^d \times I$, where $g$ and $n$ are Radon measures satisfying
$$
g \in \ms{M}(\R \times \R^d \times I), \qquad n \in \Lambda^\infty(I;\ms{M}(\R \times \R^d)).
$$
Let
\begin{equation*}
    \kappa_x \in (0,\frac{2}{m}).
\end{equation*}
If $\overline{\chi} = \int \chi \d v \in L^\infty_tL^1_x \cap L^1_{t,x}$, then
\begin{equation*}
\nrm{(\overline{\chi}_{l,j})_{(l,j) \in \N_0 \times \N_0}}_{\ell^\infty_{0,\kappa_x}(\N_0 \times \N_0;L^m_{t,x})}^m 
\lesssim \nrm{g}_{\ms{M}_{t,x,v}} + \nrm{n}_{\Lambda^\infty_v\ms{M}_{t,x}} + \nrm{\overline{\chi}}^m_{L^\infty_tL^1_x \cap L^1_{t,x}} + \nrm{\chi}_{L^1_{t,x,v}}.
\end{equation*}
\end{lemma}

\begin{proof}[Sketch of the proof]
The statement follows from a slight modification of the proof of \cite[Lemma~4.4]{GS23}.
Step 2 of this proof  can be simplified by taking the inverse of the operator 
$$
\ms{L}_v(D_{t,x}) = \p_t - D(v)\Delta_x
$$
instead of a parametrix.
In the notation of \cite{GS23} this just means that  $\ms{R}_v(D_{t,x})=0$. 
Similarly, we can take $\tilde{\ms{R}}_v(D_{t,x})=0$ and hence, 
some of the estimates can be omitted.
In particular, 
the term $\nrm{Df}_{L^{1}_{t,x,v}}$ can be neglected in
the final estimate \cite[(4.11)]{GS23}.

Let us finally comment on replacing the assumption
\begin{equation}\label{eq:assump_measure_n_L_infty}
n \in L^\infty(I;\ms{M}(\R \times \R^d))    
\end{equation}
in \cite{GS23} by  the weaker assumption
\begin{equation}\label{eq:assump_measure_n_Lambda_infty}
n \in \Lambda^\infty(I;\ms{M}(\R \times \R^d))
\end{equation}
in our setting.
By \cite[Theorem~2.29]{Pi16}, the Riesz representation theorem
$$
\ms{M}(\R \times \R^d) = [C_0(\R \times \R^d)]^*
$$
and the fact that $C_0(\R \times \R^d)$ is a separable Banach space, we have 
$$
\Lambda^\infty(I;\ms{M}(\R \times \R^d)) = 
\underline{\Lambda}^\infty(I;[C_0(\R \times \R^d)]^*).
$$
Here, given a measure space $(S,\ms{A},\mu)$ and a Banach space $X$, $\underline{\Lambda}^\infty(S;X^*)$ denotes the space of equivalence classes of weak$^{*}$ scalarly measurable functions $f:S \to X^*$ for which $s \mapsto \nrm{f(s)}_{X^*}$ belongs to $L^\infty(S)$, equipped with its natural norm 
$$
\nrm{f}_{\underline{\Lambda}^\infty(S;X^*)} := \nrm{s \mapsto \nrm{f(s)}_{X^*}}_{L^\infty(S)}.
$$
So replacing \eqref{eq:assump_measure_n_L_infty} by \eqref{eq:assump_measure_n_Lambda_infty} just means that we weaken strong measurablity to weak$^*$ scalarly measurability. The proof of \cite[Lemma~4.4]{GS23} remains valid in this case.
\end{proof}

The next lemma extends \cite[Lemma~4.5]{GS23} to our setting with one modification.

\begin{lemma}\label{lem:GS23_Lem4.3} 
Let $m \in (1,\infty)$ and let $\phi$, $\chi$, $g$, $n$ and $\overline{\chi}$ be as in Lemma~\ref{lem:GS23_Lem4.2}.
Then
\begin{equation}\label{eq:lem:GS23_Lem4.3}
\nrm{(\overline{\chi}_{l,j})_{(l,j) \in J}}_{\ell^\infty_{1,0}(J;L^1_{t,x})} \lesssim \nrm{g}_{\ms{M}_{t,x,v}} + \nrm{n}_{\Lambda^\infty_v\ms{M}_{t,x}} + \nrm{\chi}_{L^1_{t,x,v}},
\end{equation}
where 
\begin{equation*}
J:= [\{0\} \times \N_0] \cup [\N \times \N].
\end{equation*}
\end{lemma}

\begin{proof}[Sketch of the proof]
The proof of \cite[Lemma~4.5]{GS23} consists of three parts corresponding to the decomposition
\begin{align}\label{eq:decompositionNN}
\N_0 \times \N_0 = [\{0\} \times \N_0] \cup [\N \times \N] \cup [\N \times \{0\}].
\end{align}
However, we  only take over the estimates \cite[(4.19) and (4.20)]{GS23} that correspond to the cases $\{0\} \times \N_0$ and $\N \times \N$, respectively.
With this modification, \cite[Lemma~4.5]{GS23} extends to our setting which implies the lemma.
\end{proof}

Concerning the last term $\N \times \{0\}$ in the decomposition \eqref{eq:decompositionNN} we will modify the corresponding estimate \cite[(4.21)]{GS23}.
This modification is related to the observation, as made in the proof of \cite[Lemma~4.4]{GS23}, that $\overline{\chi}_{l,0}$ satisfies the equation
\begin{equation}\label{eq:j=0;equation_satisfied}
\overline{\chi}_{l,0} = -\int \phi'(v)\ms{F}_{t,x}^{-1}\frac{|\xi|^2}{\imath \tau}\ms{F}_{t,x}\chi_{l,0}\d v + \ms{F}_{t,x}^{-1}\frac{1}{\imath \tau}\hat{\varphi}_0(\xi)\ms{F}_{t,x}\int g_{l,0}\d v.    
\end{equation}
We will postpone the treatment of these Littlewood-Paley blocks to the proof of Lemma 
\ref{lem:av_lem_space-time} which will provide us more flexibility to estimate the first term in \eqref{eq:j=0;equation_satisfied}.

The next lemma is obtained by combining Lemma~\ref{lem:GS23_Lem4.2} and Lemma~\ref{lem:GS23_Lem4.3}  and is closely related to the first part of \cite[Theorem~1.2]{GST20}.

\begin{lemma}\label{lem:av_lem_space-time}
Let $m \in (1,\infty)$ and let $\phi$, $\chi$, $g$, $n$ and $\overline{\chi}$ be as in Lemma~\ref{lem:GS23_Lem4.2}.
Let $p \in (1,m]$ and define
\begin{equation*}
\kappa_t := \frac{m-p}{p}\frac{1}{m-1},\quad \kappa_x := \frac{p-1}{p}\frac{2}{m-1}.    
\end{equation*}
Let $s \in [0,1]$ and $q \in [1,p]$ and assume that $(-\Delta_x)^{s}\int \phi'\chi\, \d v \in L^q_{t,x}$. 
Then, for all $\sigma_t \in [0,\kappa_t) \cup \{0\}$ and $\sigma_x \in [0,\kappa_x)$, we have $\overline{\chi} \in W^{\sigma_t,p}_t W^{\sigma_x,p}_x$ with the corresponding norm estimate
\begin{align}
\nrm{\overline{\chi}}_{W^{\sigma_t,p}_t W^{\sigma_x,p}_x} 
&\lesssim    \nrm{g}_{\ms{M}_{t,x,v}} + \nrm{n}_{\Lambda^\infty_v\ms{M}_{t,x}} + \nrm{\overline{\chi}}_{L^\infty_tL^1_x \cap L^1_{t,x}} 
+ \nrm{\chi}_{L^1_{t,x,v}}  \nonumber \\
& \:\:+\:\: \nrm{(-\Delta_x)^{s}\int \phi'\chi\, \d v}_{L^q_{t,x}}+1. \nonumber
\end{align}
\end{lemma}

The main difference between Lemma \ref{lem:av_lem_space-time} and the first part of \cite[Theorem~1.2]{GST20} is that in the latter only the case $q=1$ and $s=0$ is considered, see also Remark \ref{rmk:assump_s;q}. 
In the proof of \cite[Theorem~1.2]{GST20} complex interpolation of Banach space-valued Besov spaces is used to obtain an interpolation inequality.
The idea to prove the general case is to perform the interpolation argument more directly  to the Littlewood-Paley decomposition of $\overline{\chi}$ restricted to $I$ by using H\"older's inequality.
In combination with a separate treatment of \eqref{eq:j=0;equation_satisfied} this will yield the stated regularity for $\overline{\chi}$ with a more flexible assumption on $\int \phi'\chi\, \d v$.  

\begin{proof}[Proof of Lemma~\ref{lem:av_lem_space-time}]
Fix $\sigma_t \in [0,\kappa_t) \cup \{0\}$ 
and $\sigma_x \in [0,\kappa_x)$. Put $\theta := \frac{p-1}{p}\frac{m}{m-1}$ and pick $s_x \in (0,2/m)$ such that
$$
s_x\theta = \frac{p-1}{p}\frac{s_xm}{m-1} \in (\sigma_x,\kappa_x). 
$$
By Lemma~\ref{lem:GS23_Lem4.2} we have  
\begin{align*}
\nrm{(\overline{\chi}_{l,j})_{(l,j) \in J}}_{\ell^\infty_{0,s_x}(J;L^m_{t,x})}
 &\lesssim \left(\nrm{g}_{\ms{M}_{t,x,v}} + \nrm{n}_{\Lambda^\infty_v\ms{M}_{t,x}} + \nrm{\overline{\chi}}^m_{L^\infty_tL^1_x \cap L^1_{t,x}} + \nrm{\chi}_{L^1_{t,x,v}} \right)^{1/m} \\
&\lesssim \nrm{g}_{\ms{M}_{t,x,v}} + \nrm{n}_{\Lambda^\infty_v\ms{M}_{t,x}} + \nrm{\overline{\chi}}_{L^\infty_tL^1_x \cap L^1_{t,x}}+\nrm{\chi}_{L^1_{t,x,v}}+1.
\end{align*}
On the other hand, the estimate \eqref{eq:lem:GS23_Lem4.3} in Lemma~\ref{lem:GS23_Lem4.3} 
gives us that
\begin{equation*}
\nrm{(\overline{\chi}_{l,j})_{(l,j) \in J}}_{\ell^\infty_{1,0}(J;L^1_{t,x})} \lesssim \nrm{g}_{\ms{M}} + \nrm{n}_{\Lambda^\infty_v\ms{M}} + \nrm{\chi}_{L^1_{t,x,v}}.
\end{equation*}
Setting $\theta=\frac{p-1}{p}\frac{m}{m-1}$ we have $1-\theta = \frac{m-p}{p(m-1)}$ and
$$
\frac{1-\theta}{1} + \frac{\theta}{m}=\frac{1}{p}, \quad \kappa_t=1(1-\theta)+0\theta, \quad s_x\theta = 0(1-\theta)+s_x\theta. 
$$
Combining the above two estimates, we thus obtain through an application of an interpolation inequality for iterated Lebesgue spaces that
\begin{align}
\nrm{(\overline{\chi}_{l,j})_{(l,j) \in J}}_{\ell^\infty_{\kappa_t,s_x\theta}(J;L^p_{t,x})}  
 &\lesssim \nrm{(\overline{\chi}_{l,j})_{(l,j) \in J}}_{\ell^\infty_{0,s_x}(J;L^m_{t,x})}
^{\theta} 
\:\cdot\:   \nrm{(\overline{\chi}_{l,j})_{(l,j) \in J}}_{\ell^\infty_{1,0}(J;L^1_{t,x})}^{1-\theta}   \nonumber \\
&\lesssim \nrm{(\overline{\chi}_{l,j})_{(l,j) \in J}}_{\ell^\infty_{0,s_x}(J;L^m_{t,x})}
+  \nrm{(\overline{\chi}_{l,j})_{(l,j) \in J}}_{\ell^\infty_{1,0}(J;L^1_{t,x})}   \nonumber \\
 &\lesssim
 \nrm{g}_{\ms{M}_{t,x,v}} + \nrm{n}_{\Lambda^\infty_v\ms{M}_{t,x}} + \nrm{\overline{\chi}}_{L^\infty_tL^1_x \cap L^1_{t,x}} 
+ \nrm{\chi}_{L^1_{t,x,v}} +1.    \label{eq:lem:av_lem_space-time;est_I}
\end{align}

Let us next consider \eqref{eq:j=0;equation_satisfied}.  
Similarly to the argument on (the bottom of) \cite[page~2466]{GST20}, by Bernstein's Lemma (see e.g.\ \cite[Lemma~2.1]{BCD11}) we have
\begin{align*}
&\nrm{\overline{\chi}_{l,0}}_{L^p_{t,x}} \\
\leq&
\nrm{\int \phi'(v)\ms{F}_{t,x}^{-1}\frac{|\xi|^2}{\imath \tau}\ms{F}_{t,x}\chi_{l,0}\d v}_{L^p_{t,x}} + \nrm{\ms{F}_{t,x}^{-1}\frac{1}{\imath \tau}\hat{\varphi}_0(\xi)\ms{F}_{t,x}\int g_{l,0}\d v}_{L^p_{t,x}} \\
\lesssim& 2^{l(\frac{1}{q}-\frac{1}{p})}\nrm{\int \phi'(v)\ms{F}_{t,x}^{-1}\frac{|\xi|^2}{\imath \tau}\ms{F}_{t,x}\chi_{l,0}\d v}_{L^q_{t,x}} + 2^{l(1-\frac{1}{p})}\nrm{\ms{F}_{t,x}^{-1}\frac{1}{\imath \tau}\hat{\varphi}_0(\xi)\ms{F}_{t,x}\int g_{l,0}\d v}_{L^1_{t,x}}.
\end{align*}
Since $|\xi|^{2-s}$ acts as a constant multiplier on the support of $\phi_0$ and $\tau^{-1}$ acts as a constant multiplier of order $2^{-l}$ on the support of $\eta_l$ and $\frac{1}{q} \leq 1$, it follows that
$$
\nrm{\overline{\chi}_{l,0}}_{L^p_{t,x}} \lesssim 2^{-l\frac{1}{p}}(\nrm{(-\Delta_x)^{s}\int \phi'\chi\, \d v}_{L^q_{t,x}} + \nrm{g}_{\ms{M}_{t,x,v}}).
$$
As 
$$
\kappa_t = \frac{m-p}{p}\frac{1}{m-1} < \frac{m-1}{p}\frac{1}{m-1} = \frac{1}{p},
$$
combining this with \eqref{eq:lem:av_lem_space-time;est_I}, we find that
\begin{align*}
&\nrm{(\overline{\chi}_{l,j})_{(l,j) \in \N_0 \times \N_0}}_{\ell^\infty_{\kappa_t,s_x\theta}(\N_0 \times \N_0;L^p_{t,x})}  \\
\lesssim&    \nrm{g}_{\ms{M}_{t,x,v}} + \nrm{n}_{\Lambda^\infty_v\ms{M}_{t,x}} + \nrm{\overline{\chi}}_{L^\infty_tL^1_x \cap L^1_{t,x}} 
+ \nrm{\chi}_{L^1_{t,x,v}} \\
& \:\:+\:\: \nrm{(-\Delta_x)^{s}\int \phi'\chi\, \d v}_{L^q_{t,x}}+1.
\end{align*}
In view of $\kappa_t > \sigma_t$ and $s_x \theta > \sigma_x$, there is the embedding
\begin{align*}
\ell^\infty_{\kappa_t,s_x\theta}(\N_0 \times \N_0;L^p(\R \times \R^d)) 
&\hra 
\ell^p_{\sigma_t,\sigma_x}(\N_0 \times \N_0;L^p(\R \times \R^d)) \\
&=
\ell^{p}_{\sigma_t}(\N_0;L^p(\R;\ell^p_{\sigma_x}(\N_0;L^p(\R^d)))).
\end{align*}
Observing that
\begin{align*}
\nrm{(\overline{\chi}_{l,j})_{(l,j) \in \N_0 \times \N_0}}_{\ell^{p}_{\sigma_t}(\N_0;L^p(\R;\ell^p_{\sigma_x}(\N_0;L^p(\R^d))))}
&= \nrm{\overline{\chi}}_{B^{\sigma_t}_{p,p}(\R;B^{\sigma_x}_{p,p}(\R^d))} 
\end{align*}
and recalling that
$$
B^{\sigma_t}_{p,p}(\R;B^{\sigma_x}_{p,p}(\R^d)) = W^{\sigma_t,p}(\R;W^{\sigma_x,p}(\R^d)),
$$
the desired result follows.
\end{proof}

\begin{lemma}\label{lem:av_lem_space-time:finite_interval}
Let $m \in (1,\infty)$ and let $\phi$ be as in Lemma~\ref{lem:GS23_Lem4.2}.
Let $\chi \in L^\infty([0,T] \times \R^d \times I)$ with $\nrm{\chi}_{L^\infty_{t,x,v}} \leq 1$ be a distributional solution to \eqref{eq:lem:GS23_Lem4.2;kinetic_eq} in the sense of distributions on $(0,T) \times \R^d \times I$, where $g$ and $n$ are Radon measures satisfying
$$
g \in \ms{M}((0,T) \times \R^d \times I), \qquad n \in \Lambda^\infty(I;\ms{M}((0,T) \times \R^d)).
$$
Let $p \in (1,m]$ and define
\begin{equation*}
\kappa_t := \frac{m-p}{p}\frac{1}{m-1},\quad \kappa_x := \frac{p-1}{p}\frac{2}{m-1}.    
\end{equation*}
Let $s \in [0,1]$ and $q \in [1,p]$ and assume that $(-\Delta_x)^{s}\int \phi'\chi\, \d v \in L^q_{t,x}$. 
Suppose that $\overline{\chi} = \int \chi \d v \in L^\infty_tL^1_x \cap L^1_{t,x}$.
Then, for all $\sigma_t \in [0,\kappa_t) \cup \{0\}$ 
and $\sigma_x \in [0,\kappa_x)$, we have $\overline{\chi} \in W^{\sigma_t,p}_t W^{\sigma_x,p}_x$ with corresponding norm estimate
\begin{align}
\nrm{\overline{\chi}}_{W^{\sigma_t,p}_t W^{\sigma_x,p}_x} 
&\lesssim    \nrm{g}_{\ms{M}_{t,x,v}} + \nrm{n}_{\Lambda^\infty_v\ms{M}_{t,x}} + \nrm{\overline{\chi}}_{L^\infty_tL^1_x \cap L^1_{t,x}}  \nonumber \\
& \:\:+\:\: \nrm{\chi}_{C_tL^1_{x,v}} + \nrm{(-\Delta_x)^{s}\int \phi'\chi\, \d v}_{L^q_{t,x}}+1. \nonumber
\end{align}
\end{lemma}

\begin{proof}
Set $\varphi_{k}(t) := \psi(kt)-\psi(kt-T)$, where $\psi \in C^\infty(\R)$ with $0 \leq \psi \leq 1$, $\supp(\psi) \subset (0,\infty)$, $\psi(t)=1$ for $t > T$ and $\nrm{\partial_t\psi}_{L^1}=1$.  
Consider $\chi_k := \varphi_k \chi$ and view it as a function of time on $\R$ instead of $[0,T]$ through extension by zero.
Then 
\begin{align*}
\ms{L}(\partial_t,\nabla_x,v)\chi_{k} 
&= \partial_t \chi_k - D(v)\Delta_x \chi_k 
= \partial_t (\varphi_k \chi) - D(v)\Delta_x (\varphi_k \chi) \\
&= \varphi_k(\partial_t \chi - D(v)\Delta_x \chi)+\partial_t\varphi_k \chi \\
&\stackrel{\eqref{eq:lem:GS23_Lem4.2;kinetic_eq}}{=} \varphi_k g + \partial_v(\varphi_k n) + \varphi_k' \chi.
\end{align*}
Thus, $\chi_k$ satisfies the corresponding equation on $\R$ with $g$ replaced by $g^k = \varphi_k g + \varphi_k' \chi$ and $n$ replaced by $n^k=\varphi_k n$.
So we can apply Lemma~\ref{lem:av_lem_space-time} to obtain 
\begin{align}
\nrm{\overline{\chi}_k}_{W^{\sigma_t,p}_t W^{\sigma_x,p}_x} 
&\lesssim    \nrm{g^k}_{\ms{M}} + \nrm{n^k}_{\Lambda^\infty_v\ms{M}} + \nrm{\overline{\chi}_k}_{L^\infty_tL^1_x \cap L^1_{t,x}} 
+ \nrm{\chi_k}_{L^1_{t,x,v}} + \nrm{\varphi_k'\chi}_{\ms{M}}  \nonumber \\
& \:\:+\:\: \nrm{(-\Delta_x)^{s}\int \phi'\chi_k\, \d v}_{L^q_{t,x}}+1, \nonumber
\end{align}
with
\begin{align*}
\nrm{\varphi_k' \chi}_{\ms{M}} &=
\nrm{\partial_t\varphi_k \chi}_{L^1_{t,x,v}} \\
&\leq \nrm{k\psi'(k\,\cdot\,) \chi}_{L^1_{t,x,v}} + \nrm{\psi'(k\,\cdot\,-T) \chi}_{L^1_{t,x,v}} \\
&\leq \nrm{k|\psi'|(k\,\cdot\,)*_t |\chi|}_{L^1_{x,v}} + \nrm{k|\psi'|(k\,\cdot\,-T)*_t |\chi|}_{L^1_{x,v}} \\
& \stackrel{k \to \infty}{\longrightarrow}
\nrm{|\chi|(0,\,\cdot\,,\,\cdot\,)}_{L^1_{x,v}} + \nrm{|\chi|(T,\,\cdot\,,\,\cdot\,))}_{L^1_{x,v}} \\
&\leq 2\norm{\chi}_{C([0,T];L^1_{x,v})}.
\end{align*}   
Taking the limit $k \to \infty$ we thus get the desired result.
\end{proof}

\begin{proof}[Proof of Theorem~\ref{thm:GST20_thm1.1}]

By Proposition~\ref{prop:mild_is_kinetic}, $u$ satisfies the kinetic form \eqref{eq:prop:mild_is_kinetic} of \eqref{eq:GPME} obtained via the kinetic function $\chi$ as defined in \eqref{eq:def_chi}, where $g=\delta_{v=u(t,x)}f$.    
Moreover, by the averaging formula \eqref{eq:averaging_formula}, we have
$$
u(t,x)=\bar{\chi}(t,x)=\int \chi(t,x,v)\d v
$$
and
$$
u(t,x)=\int \phi'(v)\chi(t,x,v)\d v.
$$
Note that $|\chi| \leq 1$, so $\nrm{\chi}_{L^\infty_{t,x,v}} \leq 1$.

Moreover, since 
$$
\nrm{u}_{L^1([0,T]\times\R^d))}
\leq T\nrm{u}_{C([0,T];L^1(\R^d))} \stackrel{\eqref{eq:prop:mild_is_kinetic;u_est_mild_sol}}{\leq}  \norm{u_0}_{L^1(\R^d)}+\norm{f}_{L^1([0,T] \times \R^d)},
$$
we have
$$
\norm{\overline{\chi}}_{L^\infty_tL^1_x \cap L^1_{t,x}} = \norm{u}_{L^\infty_tL^1_x \cap L^1_{t,x}} \lesssim \norm{u_0}_{L^1(\R^d)}+\norm{f}_{L^1([0,T] \times \R^d)}.
$$
From Lemma~\ref{lem:chi_cont_est} it follows that 
$$
\nrm{\chi}_{L^1_{t,x,v}} \leq T\nrm{\chi}_{C_t(L^1_{x,v})}\stackrel{ \eqref{eq:lem:chi_cont_est}}{=}T\nrm{u}_{C_t(L^1_{x})} \stackrel{\eqref{eq:prop:mild_is_kinetic;u_est_mild_sol}}{\leq} T(\nrm{u_0}_{L^1_x}+\nrm{f}_{L^1_{t,x}}).
$$
Furthermore, for the measures $n$ and $g$ we have \eqref{eq:prop:mild_is_kinetic;measure} and
$$
\norm{g}_{\ms{M}_{t,x,v}} \leq \norm{f}_{L^1_{t,x}}.
$$
We can thus apply Lemma~\ref{lem:av_lem_space-time:finite_interval} to obtain the desired result.
\end{proof}

\newcommand{\etalchar}[1]{$^{#1}$}

\end{document}